\documentclass{amsart}
\usepackage{hyperref}
\usepackage{amsfonts}
\usepackage{amsmath}
\usepackage{amssymb}
\usepackage{amscd}
\usepackage{graphics}
\usepackage{latexsym}
\usepackage{color}
\usepackage{epsfig}
\usepackage{amsopn}

\usepackage{booktabs}
\usepackage{longtable}
\newtheorem{theorem}{Theorem}[section]
\newtheorem{proposition}[theorem]{Proposition}
\newtheorem{lemma}[theorem]{Lemma}
\newtheorem{corollary}[theorem]{Corollary}

\theoremstyle{remark}
\newtheorem{remark}[theorem]{Remark}
\newtheorem{notation}[theorem]{Notation}

\newtheorem{observation}[theorem]{Observation}

\newtheorem{example}[theorem]{Example}

\newtheorem{definition}[theorem]{Definition}

\newcommand{\QQ}{\mathbb{Q}}

\newcommand{\FF}{\mathcal{F}}
\newcommand{\PP}{\mathbb{P}}

\newcommand{\CC}{\mathbb{C}}
\newcommand{\OO}{\mathcal{O}}

\newcommand{\ZZ}{\mathbb{Z}}

\newcommand{\bC}{{\mathbb C}}
\newcommand{\bH}{{\mathbb H}}
\newcommand{\bP}{{\mathbb P}}
\newcommand{\bQ}{{\mathbb Q}}
\newcommand{\bR}{{\mathbb R}}
\newcommand{\bZ}{{\mathbb Z}}

\newcommand{\mH}{{\mbox H}}

\newcommand{\cA}{{\mathcal A}}

\newcommand{\cD}{{\mathcal D}}
\newcommand{\cE}{{\mathcal E}}
\newcommand{\cF}{{\mathcal F}}

\newcommand{\cI}{{\mathcal I}}

\newcommand{\cM}{{\mathcal M}}
\newcommand{\cO}{{\mathcal O}}
\newcommand{\cQ}{{\mathcal Q}}

\newcommand{\ch}{\mbox{ch}}
\newcommand{\coh}{\mbox{coh}}

\newcommand{\rH}{\mbox{H}}
\newcommand{\Hom}{\mbox{Hom}}
\newcommand{\Hilb}{\mbox{Hilb}}
\newcommand{\Ext}{\mbox{Ext}}
\newcommand{\GL}{\mbox{GL}}
\newcommand{\tors}{\mbox{tors}}

\newcommand{\nt}{\noindent}

\begin{document}

\large

\title{The Minimal Model Program for the Hilbert Scheme of Points on $\PP^2$ and Bridgeland Stability}
\date{}

\author{Daniele Arcara}
\address{Department of Mathematics, Saint Vincent College, 300 Fraser Furchse Road, Latrobe, PA 15650}
\email{daniele.arcara@email.stvincent.edu}

\author{Aaron Bertram}
\address{University of Utah, Department of Mathematics, Salt Lake City, UT 84112}
\email{bertram@math.utah.edu}

\author{Izzet Coskun} 
\address{University of Illinois at
  Chicago, Department of Mathematics, Statistics and Computer Science, Chicago, IL 60607}
\email{coskun@math.uic.edu}

\author{Jack Huizenga}
\address{Harvard University, Department of Mathematics, Cambridge, MA 02138}
\email{huizenga@math.harvard.edu}

\thanks{During the preparation of this article the second author was
partially supported by the NSF grant DMS-0901128. The third author was  partially supported by the NSF grant DMS-0737581, NSF CAREER grant  0950951535, and an Arthur P. Sloan Foundation Fellowship.  The fourth author was partially supported by an NSF Graduate Research Fellowship.}


\begin{abstract}
In this paper, we study the birational geometry of the Hilbert scheme $\PP^{2[n]}$ of $n$-points on $\PP^2$. We discuss the stable base locus decomposition of the effective cone and the corresponding birational models. We give modular interpretations to the models in terms of moduli spaces of Bridgeland semi-stable objects. We construct these moduli spaces as moduli spaces of quiver representations using G.I.T. and thus show that they are projective. There is a precise correspondence between wall-crossings in the Bridgeland stability manifold and wall-crossings between Mori cones. For $n\leq 9$, we explicitly determine the walls in both interpretations and describe the corresponding flips and divisorial contractions.
\end{abstract}

\maketitle
\tableofcontents

\section{Introduction}

Let $n \geq 2$ be a positive integer. Let $\PP^{2[n]}$ denote the Hilbert scheme parameterizing zero dimensional subschemes of $\PP^2$ of length $n$. The Hilbert scheme $\PP^{2[n]}$ is a smooth, irreducible, projective variety of dimension $2n$ that contains the locus of $n$ unordered points in $\PP^2$ as a Zariski dense open subset \cite{fogarty}.  In this paper, we run the minimal model program for $\PP^{2[n]}$.  We work over the field of complex numbers $\CC$.
\medskip

The minimal model program for a parameter or moduli space $\cM$ consists of the following steps.
\begin{enumerate}
\item Determine the cones of ample and effective divisors on $\cM$ and describe the stable base locus decomposition of the effective cone.
\smallskip

\item  Assuming that the section ring is finitely generated, for every effective integral divisor $D$ on $\cM$, describe the model $$\cM(D) = \text{Proj} \left(\bigoplus_{m\geq 0} H^0(\cM, mD)\right)$$ and determine an explicit sequence of flips and divisorial contractions that relate $\cM$ to $\cM(D)$.
\smallskip

\item Finally, if possible, find a modular interpretation of $\cM(D)$.
\end{enumerate}
\medskip

Inspired by the seminal work of Birkar, Cascini, Hacon and McKernan \cite{bchm}, there has been  recent progress in understanding the minimal model program for many important moduli spaces, including the moduli space of curves (see, for example,  \cite{brendan:div}, \cite{brendan:flip}) and the Kontsevich moduli spaces of genus-zero stable maps (see, for example, \cite{stable}, \cite{mmp}). In these examples, there are three ways of obtaining different birational models of $\cM$. 
\begin{enumerate}
\item First, one may run the minimal model program on $\cM$. 
\item Second, one can vary the moduli functor. 
\item Third, since these moduli spaces are constructed by G.I.T., one can vary the linearization in the G.I.T. problem. 
\end{enumerate}
These three perspectives often produce the same models and provide three different sets of tools for understanding the geometry of $\cM$.
\medskip

In this paper, we study the birational geometry of the Hilbert scheme of points $\PP^{2[n]}$, a parameter space which plays a central role in algebraic geometry, representation theory, combinatorics, and mathematical physics (see \cite{nakajima} and \cite{gottsche1}). We discover that the birational geometry of $\PP^{2[n]}$ can also be viewed from these three perspectives.
\medskip

First, we run the minimal model program for $\PP^{2[n]}$. In Theorem \ref{dream}, we show that $\PP^{2[n]}$ is a Mori dream space. In particular, the stable base locus decomposition of $\PP^{2[n]}$ is a finite decomposition into rational polyhedral cones. Hence, in any given example one can hope to determine this decomposition completely.  In \S \ref{s-examples}, we will describe all the walls in the stable base locus decomposition of $\PP^{2[n]}$ for $n \leq 9$. By the work of Beltrametti, Sommese, G\"{o}ttsche \cite{beltrametti}, Catanese and G\"{o}ttsche \cite{catanese} and Li, Qin and Zhang \cite{li}, the ample cone of $\PP^{2[n]}$ is known. We will review the description of the ample cone of $\PP^{2[n]}$ in \S \ref{s-ample}.
\medskip

The effective cone of $\PP^{2[n]}$ is far more subtle and depends on the existence of vector bundles on $\PP^2$ satisfying interpolation.  Let $n = \frac{r(r+1)}{2}+s$ with $0 \leq s \leq r$ and assume that $\frac{s}{r}$ or $1 - \frac{s+1}{r+2}$ belongs to the set $$\Phi = \{ \alpha \ | \ \alpha > \phi^{-1} \} \cup \left\{ \frac{0}{1}, \frac{1}{2}, \frac{3}{5}, \frac{8}{13}, \frac{21}{34}, \dots \right\}, \ \ \phi = \frac{1 + \sqrt{5}}{2}, $$ where $\phi$ is the golden ratio and the fractions are ratios of consecutive Fibonacci numbers. Then, in Theorem \ref{effective},  we show that the effective cone of $\PP^{2[n]}$ is spanned by the boundary divisor $B$ parameterizing non-reduced schemes and a divisor $D_E(n)$ parameterizing subschemes that fail to impose independent conditions on sections of a Steiner bundle $E$ on $\PP^2$. The numerical conditions are needed to guarantee vanishing properties of the Steiner bundle $E$. For other $n$, we will give good bounds on the effective cone and discuss conjectures predicting the cone.
\medskip

We also introduce three families of divisors and discuss the general features of the stable base locus decomposition of $\PP^{2[n]}$. For suitable parameters, these divisors span walls of the stable base locus decomposition. As $n$ grows, the number of chambers in the decomposition grows and  the conditions defining the stable base loci become more complicated. In particular, many of the base loci consist of  loci of zero dimensional schemes of length $n$ that fail to impose independent conditions on sections of a vector bundle $E$ on $\PP^2$. Unfortunately, even when $E$ is a  line bundle $\OO_{\PP^2}(d)$, these loci are not well-understood for large $n$ and $d$. One interesting consequence of our study of the effective cone of $\PP^{2[n]}$ is a Cayley-Bacharach type theorem (Corollary \ref{cayley}) for higher rank vector bundles on $\PP^2$. 
\medskip

Second, we will vary the functor defining the Hilbert scheme.  In classical geometry, it is not at all clear how to vary the Hilbert functor. The key is to reinterpret the Hilbert scheme as a moduli space of Bridgeland semi-stable objects for a suitable Bridgeland stability condition on the heart of an appropriate t-structure on the derived category of coherent sheaves on $\PP^2$. It is then possible to vary the stability condition to obtain different moduli spaces. 
\medskip

Let $\cD^b(\operatorname{coh}(X))$ denote the bounded derived category of coherent sheaves on a smooth projective variety $X$. Bridgeland showed that the space of stability conditions on $\cD^b(\operatorname{coh}(X))$ is a complex manifold \cite{bridgeland:stable}.  We consider a complex one-dimensional slice of the stability manifold  of $\PP^2$ parameterized by an upper-half plane $s+it$, $t>0$. For each $(s,t)$ in this upper-half plane,  there is an abelian subcategory $\cA_s$ (which only depends on $s$ and not on $t$) that forms the heart of a t-structure on $\cD^b(\operatorname{coh}(\PP^2))$ and a central charge $Z_{s,t}$ such that the pair $(\cA_{s}, Z_{s,t})$ is a Bridgeland stability condition on $\PP^2$. Let $(r,c,d)$ be a fixed Chern character.  Abramovich and Polishchuk have constructed moduli stacks  $\cM_{s,t}(r,c,d)$ parameterizing Bridgeland semi-stable objects (with respect to $Z_{s,t}$) of $\cA_s$ with fixed Chern character $(r,c,d)$ \cite{abramovich}.  We show that when $s<0$ and $t$ is sufficiently large, the coarse moduli scheme of 
$\cM_{s,t}(1,0,-n)$ is isomorphic to the Hilbert scheme $\PP^{2[n]}$.
\medskip

As $t$ decreases, the moduli space $\cM_{s,t}(r,c,d)$ changes. Thus, we obtain a chamber decomposition of the $(s,t)$-plane into chambers in which the corresponding moduli spaces $\cM_{s,t}(r,c,d)$ are isomorphic. The parameters $(s,t)$ where the isomorphism class of the moduli space changes form walls called {\em Bridgeland walls}.  In \S \ref{PotentialWalls}, we determine that in the region $s<0$ and $t>0$, all the Bridgeland walls are non-intersecting, nested semi-circles with center on the real axis. 

\medskip

There is a precise correspondence between the Bridgeland walls and walls in the stable base locus decomposition. Let $x<0$ be the center of a Bridgeland wall in the $(s,t)$-plane. Let $H+ \frac{1}{2 y}B$, $y<0$,  be a divisor class spanning a wall of the stable base locus decomposition. We show that the transformation $$x = y - \frac{3}{2}$$ gives a one-to-one correspondence between the two sets of walls when $n \leq 9$ or when $x$ and $y$ are sufficiently small.  In these cases, an ideal sheaf $\cI$ lies in the stable base locus of the divisors $H + \alpha B$ for $\alpha < \frac{1}{2y}$ exactly when $\cI$ is destabilized at the Bridgeland wall with center at $x = y- \frac{3}{2}$.  One may speculate that the transformation $x= y - \frac{3}{2}$ is a one-to-one correspondence between the two sets of walls for any $n$ without any restrictions on $x$ and $y$.

\medskip

Third, in \S \ref{s-quiver} and \S \ref{s-quiver-moduli}, we will see that  the moduli spaces of Bridgeland stable objects $\cM_{s,t}(1,0,-n)$, $s<0, t>0$,  can be interpreted as moduli spaces of quiver representations and can be constructed by  Geometric Invariant Theory. In particular,  the coarse moduli schemes  of these moduli spaces are projective. There is a special region, which we call the {\em quiver region}, in the stability manifold of $\PP^2$ where the corresponding heart of the t-structure can be tilted to a category of quivers. The $(s,t)$-plane we consider intersects this region in overlapping semi-circles of radius one and center at the negative integers. Since the Bridgeland walls are all nested semi-circles  with center on the real axis, the chambers in the stability manifold all intersect the quiver region. 
Therefore, we can connect any point in the $(s,t)$-plane with $s<0$ and $t>0$ to the quiver region by a path without crossing any Bridgeland walls. We conclude that the moduli spaces of Bridgeland semi-stable objects are isomorphic to moduli spaces of quiver representations. One corollary is the finiteness of Bridgeland walls. The construction of the moduli spaces via G.I.T. also allows us to identify them with the birational models of the Hilbert scheme \cite{thaddeus}.

\medskip

The organization of this paper is as follows. In \S \ref{Hilbert-prelim}, we will recall basic facts concerning  the geometry of $\PP^{2[n]}$. In \S \ref{s-ample}, we will introduce families of effective divisors on $\PP^{2[n]}$ and recall basic facts about the ample cone of $\PP^{2[n]}$. In \S \ref{s-effective}, we will discuss the effective cone of $\PP^{2[n]}$ and the general features of the stable base locus decomposition of $\PP^{2[n]}$. In \S \ref{Bridgeland-prelim}, we will recall basic facts about Bridgeland stability conditions and introduce a complex plane worth of stability conditions that arise in our study of the birational geometry of $\PP^{2[n]}$. In \S \ref{PotentialWalls}, we study the Bridgeland walls in the stability manifold. In \S \ref{s-quiver} and \S \ref{s-quiver-moduli}, we show that in every chamber in the stability manifold, we can reach the quiver region without crossing a wall. We conclude that all the Bridgeland moduli spaces we encounter can be constructed via G.I.T. and are projective. In \S \ref{sec-actual}, we derive some useful inequalities satisfied by objects on a Bridgeland wall in the stability manifold of $\PP^2$.  Finally, in \S \ref{s-examples}, we determine the stable base locus decomposition and all the Bridgeland walls in the stability manifold for $\PP^{2[n]}$ explicitly when $n \leq 9$. 
\medskip

\noindent{\bf Acknowledgements:} We would like to thank Arend Bayer, Tom Bridgeland, Dawei Chen, Lawrence Ein, Joe Harris, Emanuele Macr{\'i}, Mihnea Popa, Artie Prendergast-Smith, and Jason Starr for many enlightening discussions.

\section{Preliminaries on the Hilbert scheme of points}\label{Hilbert-prelim}

In this section, we recall some basic facts concerning the geometry of the Hilbert scheme of points on $\PP^2$. We refer the reader to \cite{fogarty}, \cite{fogarty2} and \cite{gottsche1} for a more detailed discussion.

\begin{notation}
Let $n \geq 2$ be an integer. Let $\PP^{2[n]}$ denote the Hilbert scheme parametrizing subschemes of $\PP^2$ with  constant Hilbert polynomial $n$. Let $\PP^{2(n)}$ denote the symmetric $n$-th power of $\PP^2$ parameterizing unordered $n$-tuples of points on $\PP^2$. The symmetric product $\PP^{2(n)}$ is the quotient of the product $\PP^2 \times \cdots \times \PP^2$ of $n$ copies of $\PP^2$ under the symmetric group action $\mathfrak{S}_n$ permuting the factors.
\end{notation}

The Hilbert scheme $\PP^{2[n]}$ parametrizes subschemes $Z$  of $\PP^2$ of dimension zero with $\dim H^0(Z, \OO_Z) = n$. A subscheme $Z$ consisting of $n$ distinct, reduced points of $\PP^2$ has Hilbert polynomial $n$. Therefore, $Z$ induces a point of $\PP^{2[n]}$. The following fundamental theorem of Fogarty asserts that $\PP^{2[n]}$ is a smooth, irreducible variety and the locus parametrizing $n$ distinct, reduced points of $\PP^2$ forms a Zariski dense, open subset of $\PP^{2[n]}$.  

\begin{theorem}[Fogarty \cite{fogarty}]\label{smooth}
The Hilbert scheme  $\PP^{2[n]}$ is a smooth, irreducible, projective variety of dimension $2n$. The Hilbert scheme $\PP^{2[n]}$ admits a natural morphism to the symmetric product $\PP^{2(n)}$ called the {\em Hilbert-Chow morphism} $$h: \PP^{2[n]} \rightarrow \PP^{2(n)}.$$ The morphism $h$ is  birational and gives a crepant desingularization of the symmetric product $\PP^{2(n)}$. 
\end{theorem}
\smallskip

Theorem \ref{smooth} guarantees that every Weil divisor on $\PP^{2[n]}$ is Cartier. Hence, we can define Cartier divisors on $\PP^{2[n]}$ by imposing codimension one geometric conditions on schemes parametrized by $\PP^{2[n]}$. The Hilbert-Chow morphism allows one to compute the Picard group of $\PP^{2[n]}$. There are two natural geometric divisor classes on $\PP^{2[n]}$. 

\begin{notation}
Let $H = h^*(c_1(\OO_{\PP^{2(n)}}(1)))$ be the class of the pull-back of the ample generator from the symmetric product $\PP^{2(n)}$.  The exceptional locus of the Hilbert-Chow morphism is an irreducible divisor whose class we denote by $B$. 
\end{notation}

Geometrically, $H$  is the class of the locus of subschemes $Z$ in $\PP^{2[n]}$ whose supports intersect a fixed line $l \subset \PP^2$. Since $H$ is the pull-back of an ample divisor by a birational morphism,  $H$ is big and nef. The class $B$ is the class of the locus of non-reduced schemes. 
The following theorem of Fogarty determines the Neron-Severi space of $\PP^{2[n]}$.

\begin{theorem}[Fogarty \cite{fogarty2}]\label{pic}
The Picard group of the Hilbert scheme of points $\PP^{2[n]}$ is the free abelian group generated by $\OO_{\PP^{2[n]}}(H)$ and $\OO_{\PP^{2[n]}}(\frac{B}{2})$. The Neron-Severi space $N^1(\PP^{2[n]}) = \operatorname{Pic}(\PP^{2[n]}) \otimes \QQ$ and is spanned  by the divisor classes $H$ and $B$.
\end{theorem}

The {\em stable base locus decomposition} of a projective variety $Y$ is the partition of the the effective cone of  $Y$ into chambers according to the stable base locus of the corresponding divisors. The following theorem is the main finiteness statement concerning this decomposition for $\PP^{2[n]}$.

\begin{theorem}\label{dream}
The Hilbert scheme $\PP^{2[n]}$ is a log Fano variety. In particular, $\PP^{2[n]}$ is a Mori dream space and the stable base locus decomposition of the effective cone of $\PP^{2[n]}$ is a finite, rational polyhedral decomposition.
\end{theorem}

\begin{proof}
By \cite{bchm}, a log Fano variety is a Mori dream space. The stable base locus decomposition of the effective cone of a Mori dream space is rational, finite polyhedral \cite{hu:keel}. Hence, the only claim we need to verify is that $\PP^{2[n]}$ is log Fano. The Hilbert-Chow morphism is a crepant resolution of the symmetric product $\PP^{2(n)}$. Consequently, the canonical class of $\PP^{2[n]}$ is $K_{\PP^{2[n]}} = -3H$. Hence $-K_{\PP^{2[n]}}=3H$ is big and nef. However, $-K_{\PP^{2[n]}}= 3H$ is not ample since its intersection number with a curve in the fiber of the Hilbert-Chow morphism is zero. In Corollary \ref{nef-cone}, we will see that the divisor class $-(K_{\PP^{2[n]}} + \epsilon B) =3 H - \epsilon B$ is ample for $1 > > \epsilon >0$.  To conclude that $\PP^{2[n]}$ is log Fano, we need to know that the pair $(\PP^{2[n]}, \epsilon B)$ is klt for some small $\epsilon >0$.  Since $\PP^{2[n]}$ is smooth, as long as $\epsilon$ is chosen smaller than the log canonical threshold of $B$, the pair $(\PP^{2[n]}, \epsilon B)$ is klt. Therefore, $\PP^{2[n]}$ is a log Fano variety.
\end{proof}

\begin{remark}
More generally, Fogarty \cite{fogarty} proves that if $X$ is a smooth, projective surface, then the Hilbert scheme of points $X^{[n]}$ is a smooth, irreducible, projective variety of dimension $2n$ and the Hilbert-Chow morphism is a crepant resolution of the symmetric product $X^{(n)}$. Moreover, if $X$ is a regular surface (i.e., $H^1(X, \OO_X)=0$), then the Picard group of $X^{[n]}$ is isomorphic to Pic$(X) \times \ZZ$, hence the Neron-Severi space of $X^{[n]}$ is generated by the Picard group of $X$ and the class $B$ of the divisor of non-reduced schemes \cite[Corollary 6.3]{fogarty2}. Theorem \ref{dream} remains true with the same proof if we replace $\PP^{2[n]}$ by the Hilbert scheme of points on a del Pezzo surface. 
\end{remark}

\section{Effective divisors on  $\PP^{2[n]}$}\label{s-ample}

In this section, we introduce divisor classes on $\PP^{2[n]}$ that play a crucial role in the birational geometry of $\PP^{2[n]}$. We also recall the description of the ample cone of $\PP^{2[n]}$ studied in \cite{catanese} and \cite{li}. 

Let $k$ and $n$ be integers such that  $$\frac{k(k+3)}{2} \geq n.$$ Let $Z \in \PP^{2[n]}$ be a zero-dimensional scheme of length $n$. The long exact sequence associated to the exact sequence of sheaves
$$ 0 \rightarrow \cI_Z (k) \rightarrow \OO_{\PP^2} (k) \rightarrow \OO_Z(k) \rightarrow 0$$ 
gives rise to the  inclusion $$H^0(\PP^2, \cI_Z(k)) \rightarrow H^0(\PP^2, \OO_{\PP^2}(k)).$$  This inclusion induces a rational map to the Grassmannian $$\phi_k : \PP^{2[n]} \dashrightarrow G_k = G\left( {k+2 \choose 2}-n, {k+2 \choose 2} \right).$$ Let $D_k(n)= \phi_k^*(\OO_{G_k}(1))$ denote the pull-back of $\OO_{G_k}(1)$ by $\phi_k$.

\begin{proposition}\label{nef}
\begin{enumerate}
\item The class of $D_k(n)$ is given by 
$$D_k(n) = kH - \frac{B}{2}.$$
\item $($\cite{beltrametti}, \cite{catanese}, \cite{li} Lemma 3.8$)$  $D_k(n)$ is very ample if $k \geq n$.
\item $($\cite{beltrametti}, \cite{catanese}, \cite{li} Proposition 3.12$)$ $D_{n-1}(n)$ is base-point-free, but not ample.
\item When $k \leq n-2$, the base locus of $D_k(n)$ is contained in the locus of subschemes that fail to impose independent conditions on  curves of degree $k$. 
\end{enumerate}
\end{proposition}

\begin{proof}
Recall that a line bundle $L$ on a projective surface $S$ is called {\em $k$-very ample} if the restriction map $H^0(S, L) \rightarrow H^0(S, \OO_Z \otimes L)$ is surjective for every $Z$ in the Hilbert scheme $S^{[k+1]}$. In \cite{catanese}, Catanese and G\"{o}ttsche determine conditions that guarantee that a line bundle is $k$-very ample. When $S=\PP^2$, their results imply that $\OO_{\PP^2}(n)$ is $n$-very ample. It follows that $\phi_k$ is a morphism on $\PP^{2[n]}$ when $k \geq n-1$ and an embedding when $k \geq n$. 

Observe that $\phi_{n-1}$ is constant along the locus of schemes $Z \in \PP^{2[n]}$ that are supported on a fixed line $l$. By Bezout's Theorem, every curve of degree $n-1$ vanishing on $Z$ must vanish on $l$. Hence, $H^0(\PP^2, \cI_Z(n-1))$ is the space of polynomials of degree $n-1$ that are divisible by the defining equation of $l$. Therefore, $D_{n-1}(n)$ is base-point-free, but not ample. Assertions (2) and (3) of the proposition follow. We refer the reader to \S 3 of \cite{li} for a more detailed exposition. 

The line bundle $\OO_{G_k}(1)$ is base-point-free on the Grassmannian. Hence, the base locus of $D_{k}(n)$ is contained in the locus where $\phi_k$ fails to be a morphism. $\phi_k$ fails to be a morphism at $Z$ when the restriction map  $H^0(\PP^2, \OO_{\PP^2}(k)) \rightarrow H^0(\PP^2, \OO_Z \otimes \OO_{\PP^2}(k) )$ fails to be surjective for $Z \in \PP^{2[n]}$, equivalently when $Z$ fails to impose independent conditions on curves of degree $k$. Assertion (4) follows. 

Finally, to prove (1), we can intersect the class $D_k(n)$ with test curves. Fix a set $\Gamma$ of $n-1$ general points in $\PP^2$. Let $l_1$ be a general line in $\PP^2$. The schemes $p \cup \Gamma$ for $p \in l_1$ have Hilbert polynomial $n$ and induce a curve $R_{l_1}$ in $\PP^{2[n]}$ parametrized by $l_1$. The following intersection numbers are straightforward to compute $$R_{l_1} \cdot H =1, \ \ R_{l_1} \cdot B = 0, \ \ R_{l_1} \cdot D_{k}(n) = k.$$  Let $l_2$ be a general line in $\PP^2$ containing one of the points of $\Gamma$. Similarly, let  $R_{l_2}$ be the curve induced in $\PP^{2[n]}$ by $\Gamma$ and $l_2$. The following intersection numbers are straightforward to compute $$R_{l_2} \cdot H =1, \ \ R_{l_2} \cdot B = 2, \ \ R_{l_2} \cdot D_{k}(n) = k-1.$$   These two sets of equations determine the class of $D_{k}(n)$ in terms of the basis $H$ and $B$. This concludes the proof of the proposition. 
\end{proof}

\begin{corollary}[\cite{li}]\label{nef-cone}
The nef cone of $\PP^{2[n]}$ is the closed, convex cone bounded by the rays $H$ and $D_{n-1}(n)= (n-1)H - \frac{B}{2}$. The nef cone of $\PP^{2[n]}$ equals the base-point-free cone of $\PP^{2[n]}$.
\end{corollary}

\begin{proof}
Since the Neron-Severi space of $\PP^{2[n]}$ is two-dimensional, the nef cone of $\PP^{2[n]}$ is determined by specifying its two extremal rays. The divisor $H$ is the pull-back of the ample generator of the symmetric product $\PP^{2(n)}$ by the Hilbert-Chow morphism. Hence, it is nef and base-point-free. However, since it has intersection number zero with curves in fibers of the Hilbert-Chow morphism, it is not ample. It follows that $H$ is an extremal ray of the nef cone of $\PP^{2[n]}$. By Proposition \ref{nef} (3), $D_{n-1}(n)$ is base-point-free, hence nef, but not ample. It follows that $D_{n-1}(n)$ forms the second extremal ray of the nef cone. The base-point-free cone is contained in the nef cone and contains the ample cone (which is the interior of the nef cone). Since the extremal rays of the nef cone are base-point-free, we conclude that the nef and base-point-free cones coincide. The last fact holds more generally for extremal rays of Mori dream spaces \cite{hu:keel}. 
\end{proof}

In order to understand the birational geometry of $\PP^{2[n]}$, we need to introduce more divisors. If $n> \frac{k(k+3)}{2},$ then we have another set of divisors on $\PP^{2[n]}$ defined as follows. Let $E_k(n)$ be the class of the divisor of subschemes $Z$ of $\PP^2$ with Hilbert polynomial $n$ that have a subscheme $Z' \hookrightarrow Z$ of degree ${k+2 \choose 2}$ that fails to impose independent conditions on polynomials of degree $k$. For example, a typical point of $E_1(n)$ consists of $Z \in \PP^{2[n]}$ that have three collinear points. A typical point of $E_2(n)$ consists of subschemes $Z \in \PP^{2[n]}$ that have a subscheme of degree six supported on a conic. We will now calculate the class $E_k(n)$ by pairing it with test curves. These test curves will also play an important role in the discussion of the stable base locus decomposition.

Let $C(n)$ denote the fiber of the Hilbert-Chow morphism $h: \PP^{2[n]} \rightarrow \PP^{2(n)}$ over a general point of the diagonal. We note that $$C(n) \cdot H = 0, \ \ \ C(n) \cdot B = -2.$$ Let $r \leq n$. Let $C_r(n)$ denote the curve in $\PP^{2[n]}$ obtained by fixing $r-1$ general points on a line $l$, $n-r$ general points not contained in $l$ and a varying point on $l$. The intersection number are given by $$C_r(n) \cdot H = 1,  \ \ \ C_r(n) \cdot B = 2(r-1).$$ 
\smallskip

\begin{proposition}
The class of $E_k(n)$ is given by 
$$E_k (n)=   { n-1 \choose \frac{k(k+3)}{2}} k H  -  \frac{1}{2} { n-2 \choose \frac{k(k+3)}{2}-1} B $$
\end{proposition}
 
 \begin{proof}
 We can calculate the class $E_k(n)$ by intersecting with test curves. First, we intersect $E_k(n)$ with $C_1(n)$. We have the intersection numbers $$E_k (n)\cdot C_1(n) = {n-1 \choose \frac{k(k+3)}{2}} k, \ \ \ H \cdot C_1(n) = 1, \ \ \ B \cdot C_1(n) = 0.$$ To determine the coefficient of $B$ we intersect $E_k(n)$ with the curve $C(n)$. We have the intersection numbers $$E_k(n) \cdot C(n) = { n-2 \choose \frac{k(k+3)}{2}-1}, \ \ \ H \cdot C(n) = 0, \ \ \ B \cdot C(n) = -2.$$ The class follows from these calculations.
 \end{proof}
 
We next consider a generalization of the divisors $D_k(n)$ introduced earlier.  We begin with a definition.  
 \begin{definition}
 A vector bundle $E$ of rank $r$ on $\PP^2$ \emph{satisfies interpolation for $n$ points} if the general $Z\in \PP^{2[n]}$ imposes independent conditions on sections of $E$, i.e. if $$h^0(E\otimes \cI_Z) = h^0(E)-rn.$$
 \end{definition}

Assume $E$ satisfies interpolation for $n$ points.  In particular, we have $h^0(E)\geq rn$.  Let $W\subset H^0(E)$ be a general fixed subspace of dimension $rn$.  A scheme $Z$ which imposes independent conditions on sections of $E$ will impose independent conditions on sections in $W$ if and only if  the subspace $H^0(E\otimes \cI_Z)\subset H^0(E)$ is transverse to $W$.  Thus, informally, we obtain a divisor $D_{E,W}(n)$ described as the locus of schemes which fail to impose independent conditions on sections in $W$.  We observe that the class of $D_{E,W}(n)$ will be independent of the choice of $W$, so we will drop the $W$ when it is either understood or irrelevant to the discussion. 

\begin{remark}
We can informally interpret $D_k(n)$ as the locus of schemes $Z$ such that $Z\cup Z'$ fails to impose independent conditions on curves of degree $k$, where $Z'$ is a sufficiently general fixed scheme of degree ${k+2\choose 2}-n$.  Choosing $W = H^0(\cI_{Z'}(k))\subset H^0(\OO_{\PP^2}(k))$, we observe that $D_k(n) = D_{\OO_{\PP^2}(k)}(n)$.
\end{remark}

To put the correct scheme structure on $D_{E,W}(n)$ and compute its class, consider the universal family $\Xi_n\subset \PP^{2[n]}\times \PP^2$, with projections $\pi_1,\pi_2$.  The locus of schemes which fail to impose independent conditions on sections of $W$ can be described as the locus where the natural map $$W\otimes \OO_{\PP^{2[n]}}\to \pi_{1\ast}(\pi_2^\ast(E)\otimes \OO_{\Xi_n}) =: E^{[n]}$$ of vector bundles of rank $rn$ fails to be an isomorphism.  Consequently, it has codimension at most $1$; since the general $Z$ imposes independent conditions on sections in $W$ it is actually a divisor.  Furthermore, its class (when given the determinantal scheme structure) is just $c_1(E^{[n]})$, which can be computed using the Grothendieck-Riemann-Roch Theorem 
$$\ch(E^{[n]}) = (\pi_1)_* \left(\ch (\pi_2^* (E)\otimes \OO_{\Xi_n}) \cdot \mathrm{Td} \left(\PP^{2[n]}\times \PP^2/\PP^{2[n]}\right)\right),$$ noting that the higher pushforwards of $\pi_2^\ast E \otimes \OO_{\Xi_n}$ all vanish.  A simple calculation and the preceding discussion results in the following proposition.

\begin{proposition}\label{interpClass}
Let $E$ be a vector bundle of rank $r$ on $\PP^2$ with $c_1(E) = aL$, where $L$ is the class of a line, and suppose $E$ satisfies interpolation for $n$ points.  The divisor $D_E(n)$ has class $aH-\frac{r}{2}B$.  Furthermore, the stable base locus of the divisor class $D_E(n)$ lies in the locus of schemes which fail to impose independent conditions on sections of $E$.
\end{proposition}

We thus obtain many effective divisors that will help us understand the stable base locus  decomposition of $\PP^{2[n]}$. Since the discussion only depends on the rays spanned by these effective divisors, it is convenient to normalize their expressions so that the coefficient of $H$ is one. The divisors $D_k(n)$ give the rays $$H - \frac{B}{2k}.$$
The divisors $E_k(n)$ give the rays $$H - \frac{k+3}{4(n-1)} B.$$
Furthermore, whenever there exists a vector bundle $E$ on $\PP^2$ of rank $r$ with $c_1(E)=a L$ satisfying interpolation for $n$ points we have the divisor $D_E(n)$ spanning the ray $$H - \frac{r}{2a} B.$$

\section{The stable base locus decomposition of the effective cone of $\PP^{2[n]}$}\label{s-effective}

In this section, we collect facts about the effective cone of the Hilbert scheme $\PP^{2[n]}$ and study the general features of the stable base locus decomposition of $\PP^{2[n]}$. 
\medskip

Fix a point $p \in \PP^2$.  Let $U_n (p) \subset \PP^{2[n]}$ be the open subset parametrizing schemes whose supports do not contain the point $p$.  There is an embedding of $U_n(p)$ in $\PP^{2[n+1]}$ that associates to the scheme $Z \in U_n(p)$ the scheme $Z \cup p$ in $\PP^{2[n+1]}$. The induced rational map 
$$i_p: \PP^{2[n]} \dashrightarrow \PP^{2[n+1]}$$ gives rise to a homomorphism $$i_p^*: \mbox{Pic} (\PP^{2[n+1]}) \rightarrow \mbox{Pic}(\PP^{2[n]}).$$ Observe that $i_p^*(H) = H$ and $i_p^*(B)=B$. Hence, the map $i_p^*$ does not depend on the point $p$ and gives an isomorphism between the Picard groups. 

For the purposes of the next lemma, we identify the Neron-Severi space $N^1(\PP^{2[n]})$ with the vector space spanned by two basis elements labelled $H$ and $B$. Let $\mbox{Eff}(\PP^{2[n]})$ denote the image of the effective cone of $\PP^{2[n]}$ under this identification. We can thus view the effective cones of $\PP^{2[n]}$ for different $n$ in the same vector space.  Under this identification, we have the following inclusion.

\begin{lemma}\label{inclusion}
$\operatorname{Eff}(\PP^{2[n+1]}) \subseteq \operatorname{Eff}(\PP^{2[n]})$.
\end{lemma}

\begin{proof}
Let $D$ be an effective divisor on $\PP^{2[n+1]}$. Then $i_p^*(D)$ is a divisor class on $\PP^{2[n]}$. Since $i_p^*(H) = H$ and $i_p^*(B)=B$, under our identification, $D$ and $i_p^*(D)$ represent the same point in the vector space spanned by $H$ and $B$. The proof of the lemma is complete if we can show that $i_p^*(D)$ is the class of an effective divisor.  Let $p_1$ be a point such that there exists a scheme consisting of $n+1$ distinct, reduced points $p_1, \dots, p_{n+1}$ not contained in $D$. Therefore, $i_{p_1}(U_n(p_1)) \not\subset D$ and $D \cap i_{p_1}(U_n(p_1))$ is an effective divisor on $i_{p_1}(U_n(p_1))$. Since $i_p^*(D)$ does not depend on the choice of point $p$, we conclude that  $i_p^*(D)$ is the class of an effective divisor on $\PP^{2[n]}$. 
\end{proof}

\begin{remark}
Let $n_1 < n < n_2$. By Lemma \ref{inclusion}, if we know $\operatorname{Eff}(\PP^{2[n_1]})$ and $\operatorname{Eff}(\PP^{2[n_2]})$, then we bound $\operatorname{Eff}(\PP^{2[n]})$ both from above and below.
\end{remark}

We saw in Proposition \ref{interpClass} that if $E$ is a vector bundle of rank $r$ and $c_1(E) = aL$ satisfying interpolation for $n$ points, then we get an effective divisor on $\PP^{2[n]}$ with class $aH - \frac{r}{2}B$. In view of this, it is important to understand vector bundles that satisfy interpolation for $n$ points.
Recall from the introduction the definition of the set $\Phi$: $$\Phi = \{ \alpha \ | \ \alpha > \phi^{-1} \} \cup \left\{ \frac{0}{1}, \frac{1}{2}, \frac{3}{5}, \frac{8}{13}, \frac{21}{34}, \dots \right\}, \ \ \phi = \frac{1 + \sqrt{5}}{2}, $$ where $\phi$ is the golden ratio and the fractions are ratios of consecutive Fibonacci numbers. 
The following theorem of the fourth author guarantees the existence of vector bundles satisfying interpolation for $n$ points.

\begin{theorem}\cite[Theorem 4.1]{jack:thesis}\label{interpolation}
Let $$n = \frac{r(r+1)}{2} + s, \ \ s \geq 0.$$ Consider a general vector bundle $E$ given by the resolution $$0 \rightarrow \OO_{\PP^2}(r-2)^{\oplus ks} \rightarrow \OO_{\PP^2}(r-1)^{\oplus k(s+r)} \rightarrow E \rightarrow 0.$$
For sufficiently large $k$, $E$ is a vector bundle that satisfies interpolation for $n$ points if and only if $\frac{s}{r} \in \Phi$.
Similarly, let $F$ be a general vector bundle given by the resolution $$0 \rightarrow F \rightarrow \OO_{\PP^2}(r)^{\oplus k(2r-s+3)} \rightarrow \OO_{\PP^2}(r+1)^{\oplus k(r-s+1)} \rightarrow 0.$$ For sufficiently large $k$, $F$ has interpolation for $n$ points if and only if $ 1 - \frac{s+1}{r+2} \in \Phi$.
 \end{theorem}

\begin{remark}
In Theorem \ref{interpolation}, the rank of $E$ is $kr$ and $c_1(E) = k (r^2 - r + s)L$. Hence, the corresponding effective divisor $D_E(n)$ on $\PP^{2[n]}$ lies on the ray $H - \frac{r}{2(r^2-r+s)}B$. Similarly, the rank of $F$ is $k(r+2)$ and $c_1(F) = k (r^2 + r+s-1)L$. Hence, the divisor $D_F(n)$ on $\PP^{2[n]}$ lies on the ray $H - \frac{r+2}{2(r^2 +r+s-1)}B.$
\end{remark}

A consequence of Theorem \ref{interpolation} is the following theorem that determines the effective cone of $\PP^{2[n]}$ in a large number of cases.

\begin{theorem}\label{effective}
Let $$n = \frac{r(r+1)}{2} + s, \ \ 0 \leq s \leq r.$$
\begin{enumerate}
\item If $\frac{s}{r} \in \Phi$, then the effective cone of $\PP^{2[n]}$ is the closed cone bounded by the rays $$H - \frac{r}{2(r^2 -r+s)}B\ \  \mbox{and} \ \ B.$$
\item If  $1 - \frac{s+1}{r+2} \in \Phi$ and $s \geq 1$, then the effective cone of $\PP^{2[n]}$ is the closed cone bounded by the rays $$H - \frac{r+2}{2(r^2+r+s-1)}B \ \ \mbox{and} \ \ B.$$
\end{enumerate}
\end{theorem}

\begin{proof}
The proof of this theorem relies on the idea that the cone of moving curves is dual to the effective cone. Recall that a curve class $C$ on a variety $X$ is called moving if irreducible curves in the class $C$ cover  a Zariski open set in $X$. If $C$ is a moving curve class and $D$ is an effective divisor, then $C \cdot D \geq 0$. Thus each moving curve class gives a bound on the effective cone. To determine the effective cone, it suffices to produce two rays spanned by effective divisor classes and corresponding moving curves that have intersection number zero with these effective divisors. 

The divisor class $B$ is the class of the exceptional divisor of the Hilbert-Chow morphism. Consequently, it is effective and extremal. Alternatively, to see that $B$ is extremal, notice that $C_1(n)$ is a moving curve with $C_1(n)\cdot B = 0$.  Hence, $B$ is an extremal ray of the effective cone.

The other extremal ray of the effective cone is harder to find. Observe that when $\frac{s}{r} \in \Phi$, Theorem \ref{interpolation} constructs an effective divisor $D_E(n)$ along the ray $H - \frac{r}{2(r^2-r + s)} B.$ Hence, the effective cone contains the cone generated by $B$ and $H - \frac{r}{2(r^2 -r + s)} B.$

Let $Z$ be a general scheme of dimension zero and length $n$.  The dimension of the space of curves of degree $r$ in $\PP^2$ is $\frac{r(r+3)}{2}$. Since a general collection of simple points impose independent conditions on curves of degree $r$ and $$\frac{r(r+3)}{2} - \frac{r(r+1)}{2} -s = r-s \geq 0,$$ $Z$ is contained in a smooth curve $C$ of degree $r$. The scheme $Z$ defines a divisor $D_Z$ on the smooth curve $C$ of degree $n$. The genus of $C$ is $\frac{(r-1)(r-2)}{2}$. Therefore, by the Riemann-Roch Theorem $$h^0(C, \OO_C(D_Z)) \geq \frac{r(r+1)}{2} +s - \frac{(r-1)(r-2)}{2} + 1= 2r+s \geq 2.$$ Let $P$ be a general pencil in $H^0(C, \OO_C(D_Z))$ containing $Z$. The corresponding divisors of degree $n$ on $C$ induce a curve $R$ in the Hilbert scheme $\PP^{2[n]}$. We then have the following intersection numbers: $$R \cdot H = r, \ \ \mbox{and} \ \ R \cdot B = 2(r^2-r+s).$$ The first intersection number is clear since it equals the degree of the curve $C$. The second intersection number can be computed using the Riemann-Hurwitz formula. It is easy to see that $R \cdot B$ is the degree of the ramification divisor of the map $\phi_P : C \rightarrow \PP^1$ induced by the pencil $P \subset H^0(C, \OO_C(D_Z))$. The Riemann-Hurwitz formula implies that this degree is $$2n + (r-1)(r-2) -2 = 2 (r^2-r+s). $$ We conclude that $R \cdot D_E(n) = 0$. Since $Z$ was a general point of $\PP^{2[n]}$ and we constructed a curve in the class $R$ containing $Z$, we conclude that $R$ is a moving curve class. Therefore, the effective divisor $D_E(n)$ is extremal in the effective cone. We deduce that the effective cone is equal to the cone spanned by $B$ and $H - \frac{r}{2(r^2 -r + s)} B.$

When $1- \frac{s+1}{r+2} \in \Phi$, then by Theorem \ref{interpolation}, $D_F(n)$ is an effective divisor along the ray $H- \frac{r+2}{2(r^2 + r + s -1)}B$. Therefore, the effective cone contains the cone generated by  the rays $H- \frac{r+2}{2(r^2 + r + s -1)}B$ and $B$. 

Let $Z$ be a general scheme of dimension zero and length $n$. The dimension of the space of curves of degree $r+2$ in $\PP^2$ is $\frac{(r+2)(r+5)}{2}$. Since a general collection of simple points impose independent conditions on curves of degree $r+2$ and  $$\frac{(r+2)(r+5)}{2} - \frac{r(r+1)}{2} -s = 3r+5-s \geq 0,$$ $Z$ is contained in a smooth curve $C$ of degree $r+2$. The scheme $Z$ defines a divisor $D_Z$ of degree $n$ on $C$. The genus of $C$ is $\frac{r(r+1)}{2}$. By the Riemann-Roch Theorem, 
$$h^0(C, \OO_C(D_Z)) \geq \frac{r(r+1)}{2} +s - \frac{r(r+1)}{2} + 1= s+1 \geq 2, $$  provided $s \geq 1$. Let $P$ be a general pencil in  $H^0(C, \OO_C(D_Z))$ containing $Z$. The corresponding divisors of degree $n$ on $C$ induce a curve $R$ in the Hilbert scheme $\PP^{2[n]}$. We then have the following intersection numbers: $$R \cdot H = r+2, \ \ \mbox{and} \ \ R \cdot B =2(r^2 + r + s -1).$$ Here the second number is equal to the degree of the ramification divisor of the map $\phi_P: C \rightarrow \PP^1$ induced by the pencil $P \subset H^0(C, \OO_C(D_Z))$ and is computed by the Riemann-Hurwitz formula. We conclude that $R \cdot D_F(n) = 0$. Since we constructed a curve in the class $R$ containing a general point $Z$, we conclude that $R$ is a moving curve class. Therefore, the effective divisor $D_F(n)$ is extremal in the effective cone. We deduce that the effective cone is equal to the cone spanned by $B$ and $H - \frac{r+2}{2(r^2 + r + s-1)} B.$ This concludes the proof of the theorem.
\end{proof}

\begin{remark}
Several special cases of Theorem \ref{effective} are worth highlighting since the divisors have more concrete descriptions. 
\smallskip

\noindent {$\bullet$} When $n = \frac{r(r+1)}{2}$, then the divisor $E_{r-1}(n)$ parameterizing zero dimensional schemes of length $n$ that fail to impose independent conditions on sections of $\OO_{\PP^2}(r-1)$ is an effective divisor on the extremal ray  $H - \frac{B}{2(r-1)}$. Hence, the effective cone is the cone generated by $H - \frac{B}{2(r-1)}$ and $B$.
\smallskip

\noindent $\bullet$ When $n= \frac{r(r+1)}{2}-1$, by Lemma \ref{inclusion}, the effective cone contains the cone generated by $H - \frac{B}{2(r-1)}$ and $B$. By Theorem \ref{effective}, the effective cone is equal to the cone spanned by $H- \frac{B}{2(r-1)}$ and $B$. Hence, the extremal ray of the effective cone in this case is generated by the pull-back of $E_{r-1}(n+1)$ under the rational map $i_p: \PP^{2[n]} \dashrightarrow \PP^{2[n+1]}$. 
\smallskip

\noindent $\bullet$ When $n = \frac{r(r+1)}{2} + 1$, then the divisor $E_{r-1}(n)$ parameterizing zero-dimensional subschemes of length $n$ that have a subscheme of length $n-1$ that fails to impose independent conditions on sections of $\OO_{\PP^2}(r-1)$ is an effective divisor on the ray $H - \frac{r+2}{4(n-1)}B$. Hence, the effective cone is the cone generated by $E_{r-1}(n)$ and $B$.
\smallskip

\noindent $\bullet$ When $n = \frac{(r+1)(r+2)}{2}-2=\frac{r(r+1)}{2}+(r-1)$ with $r\geq 3$, a general collection of $n$ points lies on a pencil of curves of degree $r$.  The base locus of this pencil consists of $r^2$ points, and we obtain a rational map $\PP^{2[n]}\dashrightarrow \PP^{2[r^2-n]}$ sending $n$ points to the $r^2-n$ points residual to the $n$ in the base locus of this pencil.  The pull-back of $H$ under this map gives an effective divisor spanning the extremal ray $H - \frac{r}{2(r^2-1)}B$.

\noindent $\bullet$ When $n = \frac{r(r+1)}{2} + \frac{r}{2}$ with $r$ even, we can take the vector bundle in the construction of the extremal ray of the effective cone to be a twist of the tangent bundle $T_{\PP^2}(r-2)$. The corresponding divisor lies on the extremal ray $H - \frac{1}{2r-1} B$. In this case, the effective cone is the cone generated by $H - \frac{1}{2r-1} B$ and $B$.  
\end{remark}

\begin{remark}
Theorem \ref{effective} determines the effective cone of $\PP^{2[n]}$ for slightly more than three quarters of all values of $n$. In order to extend the theorem to other values of $n= \frac{r(r+1)}{2} + s$, one has to consider interpolation for more general bundles. For example, one may consider bundles of the form $$0 \rightarrow \OO_{\PP^2}(r-3)^{\oplus s} \rightarrow \OO_{\PP^2}(r-1)^{\oplus 2r+s-1} \rightarrow E \rightarrow 0.$$ Provided that $2s<r$ and either $$ \sqrt{2}-1 <\frac{s}{r - \frac{1}{2}}  \ \ \mbox{or} \ \ \frac{s}{r - \frac{1}{2}} $$  is a convergent of the continued fraction expansion of $\sqrt{2}-1$, it is reasonable to expect that $E$ satisfies interpolation and the divisor $D_E(n)$ spans the extremal ray of the effective cone. Dually, consider the bundle $$0 \rightarrow F \rightarrow \OO_{\PP^2}(r)^{\oplus 3r-s+6} \rightarrow \OO_{\PP^2}(r+2)^{\oplus r-s+1} \rightarrow 0.$$ Provided that $2s>r$ and either $$\frac{r-s+1}{r+\frac{5}{2}} > \sqrt{2}-1 \ \ \mbox{or}  \ \ \frac{r-s+1}{r+\frac{5}{2}}$$ is a convergent of the continued fraction expansion of $\sqrt{2}-1$, one expects $F$ to satisfy interpolation and the corresponding divisor $D_F(n)$ to span the extremal ray of the effective cone of $\PP^{2[n]}$. Unfortunately, at present interpolation does not seem to be known for these bundles. 
\end{remark}

While the bounds may not be sharp for $n$ not covered by Theorem \ref{effective}, the proof still shows the following corollary. 

 \begin{corollary}[A Cayley-Bacharach Theorem for higher rank vector bundles on $\PP^2$]\label{cayley}
Let $n=\frac{r(r+1)}{2} + s$ with $0 \leq s \leq r$. 
\begin{enumerate}
\item If $\frac{s}{r} \geq \frac{1}{2}$, then the effective cone of $\PP^{2[n]}$ is contained in the cone generated by $H - \frac{r}{2(r^2-r+s)}B$ and $B$.
\item If $\frac{s}{r} < \frac{1}{2}$, then the effective cone of $\PP^{2[n]}$ is contained in the cone generated by $H - \frac{r+2}{2(r^2+r+s-1)}B$ and $B$.
\end{enumerate}
Let $E$ be a vector bundle on $\PP^2$ with rank $k$ and $c_1(E)= aL$. If 
$$\frac{k}{a} > \frac{r}{r^2 - r + s} \ \  \mbox{when} \ \  \frac{s}{r} \geq \frac{1}{2}, \ \  \mbox{or}  \ \ 
\frac{k}{a} > \frac{r+2}{r^2 +r + s-1}\ \  \mbox{when}  \ \ \frac{s}{r} < \frac{1}{2},$$ 
then $E$ cannot satisfy interpolation for $n$ points.  That is, every $Z\in \PP^{2[n]}$ fails to impose independent conditions on sections of $E$.
\end{corollary}

We now turn our attention to describing some general features of the stable base locus decomposition of $\PP^{2[n]}$. 

\begin{notation}
We denote the closed cone in the Neron-Severi space generated by two divisor classes $D_1, D_2$ by $[D_1, D_2]$. We denote the interior of this cone by $(D_1, D_2)$. We use $[D_1, D_2)$ and $(D_1, D_2]$ to denote the semi-closed cones containing the ray spanned by $D_1$ and, respectively, $D_2$.
\end{notation}

\begin{proposition}\label{BH}
The cone $(H,B]$ forms a chamber of the stable base locus decomposition of $\PP^{2[n]}$. Every divisor in this chamber has stable base locus $B$.
\end{proposition}

\begin{proof}
Since the divisor class $H$ is the pull-back of the ample generator from the symmetric product $\PP^{2(n)}$, it is base point free. Every divisor $D$ in the chamber $(H, B]$ is a non-negative linear combination  $aH+ bB$. Hence, the base locus of $D$ is contained in $B$. On the other hand, $C(n) \cdot D = -2b <0$, provided that $b>0$. Since curves in the class $C(n)$ cover the divisor $B$,  $B$ has to be in the base locus of $D$ for every $D \in (H,B]$. We conclude that $B$ is the stable base locus of every divisor in the chamber $(H,B]$.
\end{proof}

\begin{corollary}
The birational model of $\PP^{2[n]}$ corresponding to the chamber $[H,B)$ is the symmetric product $\PP^{2(n)}$.
\end{corollary}

\begin{proof}
If $D \in [H, B)$, then $D= m H + \alpha B$ with $m>0$ and $\alpha \geq 0$. By proposition \ref{BH}, the  moving part of $D$ is $mH$. Since $H$ induces the Hilbert-Chow morphism, we conclude that the birational model corresponding to $D$ is the symmetric product $\PP^{2(n)}$. 
\end{proof}

Together with our earlier description of the nef cone in Proposition \ref{nef}, this completes the description of the stable base locus decomposition in the cone spanned by $H - \frac{1}{2n-2}B$ and $B$.  To help determine the chambers beyond the nef cone, a couple simple lemmas will be useful.

\begin{lemma}\label{stableContainLemma}
Suppose $0<\alpha < \beta$.  The stable base locus of $H-\alpha B$ is contained in the stable base locus of $H-\beta B$.
\end{lemma}
\begin{proof}
The divisor $H$ is base-point free.  There is a positive number $c$ such that $H-\beta B + cH$ lies on the ray spanned by $H-\alpha B$.  Thus any point in the stable base locus of $H-\alpha B$ will also be in the stable base locus of $H-\beta B$.
\end{proof}

\begin{lemma}\label{chamberCurveLemma}
Let $C$ be a curve class in $\PP^{2[n]}$ with $C\cdot H > 0$, and suppose $C\cdot (H-\alpha B)=0$ for some $\alpha>0$.  Then the stable base locus of every divisor $H-\beta B$ with $\beta >\alpha$  contains the locus swept out by irreducible curves of class $C$.
\end{lemma}
\begin{proof}
Since $C\cdot H>0$ and $C\cdot (H-\alpha B)=0$, we have $C\cdot B>0$, and thus $C\cdot (H-\beta B)<0$ for every $\beta > \alpha$.  Thus any effective divisor on the ray $H-\beta B$ contains every irreducible curve of class $C$.
\end{proof}

To finish the section, we wish to describe several chambers of the stable base locus decomposition arising from divisors of the form $D_k(n)$.  The following lemma will play a key role in identifying the stable base loci.

\begin{lemma}\label{2dplus1lemma}
Let $Z$ be a zero-dimensional scheme of length $n$.  
\begin{enumerate}
\item[(a)] If $n\leq 2d+1$,  then $Z$ fails to impose independent conditions on curves of degree $d$ if and only if it has a collinear subscheme of length at least $d+2$.

\item[(b)] Suppose $n = 2d+2$ and $d\geq 2$.  Then $Z$ fails to impose independent conditions on curves of degree $d$ if and only if it either has a collinear scheme of length at least $d+2$ or it is contained on a (potentially reducible or nonreduced) conic curve.
\end{enumerate}
\end{lemma}
\begin{proof}
(a) It is clear that if $Z$ has a collinear subscheme of length at least $d+2$ then $Z$ fails to impose independent conditions on curves of degree $d$. 
For the more difficult direction, we proceed by induction on $d$.  Suppose that $Z$ has no collinear subscheme of length $d+2$.  Given a line $L$ in $\mathbb{P}^2$ we can consider the \emph{residuation sequence} \begin{equation}\label{residuation} 0\to \mathcal I_{Z'}(d-1)\stackrel{L}{\to} \mathcal I_{Z}(d) \to \mathcal I_{Z\cap L\subset L}(d)\to 0\end{equation} where $Z'$ is the subscheme of $Z$ defined by the ideal quotient $(\mathcal I_Z:\mathcal I_L)$.  Clearly the scheme $Z\cap L$ imposes independent conditions on curves of degree $d$ since $Z$ contains no collinear subscheme of length $d+2$.  Thus $H^1(\mathcal I_{Z\cap L\subset L}(d)) = 0$. If we show that $Z'$ imposes independent conditions on curves of degree $d-1$, then it will follow that $H^1(\mathcal I_{Z'}(d-1))=0$ and hence $Z$ imposes independent conditions on curves of degree $d$.

To apply our induction hypothesis, choose the line $L$ such that the intersection $Z\cap L$ has as large a length $\ell$ as possible.  Clearly $\ell \geq 2$ unless $Z$ is just a point, so $Z'$ has length at most $2(d-1)+1$.  If  $\ell\leq d$, then since $Z$ contains no collinear subscheme of length at least $d+1$ we find $Z'$ also contains no collinear subscheme of length at least $d+1$, so by the induction hypothesis $Z'$ imposes independent conditions on curves of degree $d-1$.  On the other hand, if $\ell =d+1$, then $Z'$ has degree $d$, so does not contain a collinear subscheme of degree $d+1$.

(b) Again it is clear that if $Z$ is contained in a conic or has a subscheme of length $d+2$ supported on a line that $Z$ fails to impose independent conditions on curves of degree $d$.

Suppose $Z$ does not lie on a conic and that it does not contain a collinear subscheme of length $d+2$.  We will reduce to part (a) by choosing an appropriate residuation depending on the structure of $Z$.  

If $Z$ meets a line in a scheme of length $d+1$, then the subscheme $Z'$ of length $d+1$ residual to this line cannot also lie on a line or $Z$ would lie on a conic.  Then $Z'$ imposes independent conditions on curves of degree $d+1$ by part (a), so by the residuation sequence $Z$ imposes independent conditions on curves of degree $d+2$.  

Next, suppose $Z$ does not meet a line in a scheme of length $d+1$ but $Z$ has a collinear subscheme of length at least $3$.  Looking at the scheme $Z'$ residual to this line, we may again apply part (a) to conclude $Z'$ imposes independent conditions on curves of degree $d+1$, and thus that $Z$ imposes independent conditions on curves of degree $d+2$.

Finally assume no line meets $Z$ in a scheme of length greater than $2$. Choose any length $5$ subscheme $Z''$ of $Z$, and let $C$ be a conic curve containing $Z''$.  The curve $C$ is in fact reduced and irreducible since $Z$ contains no collinear triples.  Since $Z$ does not lie on a conic, $Z\cap C$ is a subscheme of $C$ of length at most $2d+1$.  But any subscheme of $C$ of length at most $2d+1$ imposes independent conditions on sections of $\OO_{C}(d) = \OO_{\PP^1}(2d)$.  Thus if $Z'$ is residual to $Z\cap C$ in $Z$, we see by the residuation sequence corresponding to $C$ that $Z$ imposes independent conditions on curves of degree $d$ if $Z'$ imposes independent conditions on curves of degree $d-2$.  Since $Z'$ has degree at most $(2d+2)-5 = 2(d-2)+1$ and contains no collinear triples, we conclude that it imposes independent conditions on curves of degree $d-2$ by part (a).
\end{proof}

We now identify many of the chambers in the stable base locus decomposition.

\begin{proposition}\label{DkBase}
Let $n\leq k(k+3)/2$, so that $D_k(n)$ is an effective divisor on $\PP^{2[n]}$.  Its class lies on the ray $H-\frac{1}{2k}B$.
\begin{enumerate}
\item[(a)] If $n\leq 2k+1$, the stable base locus of divisors in the chamber $[H-\frac{1}{2k} B,H-\frac{1}{2k+2}B)$ consists of schemes of length $n$ with a linear subscheme of length $k+2$.

\item[(b)] If $n = 2k+2$, the stable base locus of divisors in the chamber $[H-\frac{1}{2k}B,H-\frac{1}{2k+2}B)$ contains the locus of schemes of length $n$ with a linear subscheme of length $k+2$ and is contained in the locus of schemes of length $n$ which either have a linear subscheme of length $k+2$ or which lie on a conic.

\item[(c)] In any case, the ray $H-\frac{1}{2k} B$ spans a wall in the stable base locus decomposition of the effective cone of $\PP^{2[n]}$.
\end{enumerate}
\end{proposition}
\begin{proof}
Recall that $C_{k}(n)$ is the curve class in $\PP^{2[n]}$ given by fixing $k-1$ points on a line, $n-k$ points off the line, and letting a final point move along the line.  We have $C_k(n)\cdot H = 1$ and $C_k(n) \cdot B = 2(k-1)$.  It follows that $D_{k+1}(n)\cdot C_{k+2}(n) = 0$ for all $k$.  By Lemma \ref{chamberCurveLemma}, if $\alpha>\frac{1}{2k+2}$ then the locus swept out by irreducible curves of class $C_{k+2}(n)$ is contained in the stable base locus of $H-\alpha B$.  This locus certainly contains the locus of schemes with a linear subscheme of length $k+2$.  On the other hand, by Proposition \ref{nef} the divisor $D_k(n)$ has stable base locus contained in the locus of schemes of length $n$ which fail to impose independent conditions on curves of degree $k$. 

By Lemma \ref{2dplus1lemma} (a), we see that if $n\leq 2k+1$ then the stable base locus of $D_k(n)$ is contained in the locus of schemes of length $n$ with a linear subscheme of length $k+2$, and therefore that the stable base locus of divisors in the chamber $[H-\frac{1}{2k}B,H-\frac{1}{2k+2}B)$ is precisely this locus.

The conclusion for (b) follows similarly by using Lemma \ref{2dplus1lemma} (b) along with Lemma \ref{stableContainLemma}.

For (c), notice that every divisor $H-\alpha B$ with $\alpha > \frac{1}{2k}$ has the locus of schemes of length $n$ with a linear subscheme of length $k+1$ in its stable base locus.  Let $Z$ be a general scheme of length $n$ with a linear subscheme of length $k+1$.  Using the residuation sequence (\ref{residuation}),  we see that $Z$ imposes independent conditions on curves of degree $k$.  By Proposition \ref{nef}, $Z$ does not lie in the base locus of $D_k(n)$, so the ray spanned by $D_k(n)$ forms a wall in the stable base locus decomposition.
\end{proof}

\section{Preliminaries on Bridgeland stability conditions}\label{Bridgeland-prelim}

In this section, we review the basic facts concerning Bridgeland stability conditions introduced in \cite{bridgeland:stable} and recall several relevant constructions from \cite{bridgeland:K3} and \cite{aaron:k3}.
\smallskip

Let $D^b(\coh(\bP^2))$ denote the bounded derived category of coherent sheaves on $\PP^2$. 

\begin{definition}
A {\em pre-stability condition} on $\bP^2$ consists of a triple $(\cA; d,r)$ such that:
\begin{itemize}
\item $\cA$ is the heart of a $t$-structure on $D^b(\coh(\bP^2))$. 
\item $r$ and $d$ are linear maps: $$r,d: K(\cD(\coh(\bP^2))) \rightarrow \bR$$
from the $K$-group of the derived category to $\bR$ satisfying: 
\begin{itemize}
\item[(*)] $r(E) \ge 0$ for all $E\in \cA$, and 
\item[(**)] if $r(E) = 0$ and $E \in \cA$ is nonzero, then $d(E) > 0$.
\end{itemize}
\end{itemize}
A pre-stability condition is a {\it stability condition} if objects of $\cA$ all have the {\em Harder-Narasimhan property}, which we now define.
\end{definition}

\begin{remark}
The function $Z = -d + ir$ is called the {\em central charge} and maps non-zero objects of $\cA$ to the upper-half plane $\{ \rho e^{i\pi \theta} \ | \ \rho >0, 0 < \theta \leq 1\}$.
\end{remark}

\begin{definition}
 The {\it slope} of a nonzero object $E \in \cA$ (w.r.t. $(r,d)$) is:

$$\mu(E) := \left\{ \begin{array}{l} {d(E)}/{r(E)} \ \mbox{if $r(E) \ne 0$} \\ \\ +\infty \ \mbox{if $r(E) = 0$}
\end{array}\right.$$
\end{definition}

\begin{definition}
 $E \in \cA$ is {\it stable} (resp. {\it semi-stable}) if
$$F \subset E  \Rightarrow \mu(F) < \mu(E) \ \mbox{(resp. $\mu(F) \le \mu(E)$)}$$
for all nonzero proper subobjects $F \subset E$ in the category $\cA$. 
\end{definition}

\begin{definition}
A pre-stability condition $(\cA;d,r)$ has the Harder-Narasimhan property if
every nonzero object $E \in \cA$ admits
a finite filtration:
$$0 \subset E_0 \subset E_1 \subset \dots \subset E_n = E$$
uniquely determined by the property that each $F_i := E_{i}/E_{i-1}$ is semi-stable and 
$\mu(F_1) > \mu(F_2) > \dots > \mu(F_{n})$. This filtration is called the {\em Harder-Narasimhan filtration} of $E$.
\end{definition} 

\begin{remark}
 Let $L$ denote the hyperplane class on $\PP^2$. The standard $t$-structure with ``ordinary'' degree and rank
$$d(E) := \ch_1(E)\cdot L, \ \ r(E) := \ch_0(E)\cdot L^2$$
on coherent sheaves on $\bP^2$ are {\bf not} a pre-stability condition because:
$$r(\bC_p) =  0 = d(\bC_p)$$
for skyscraper sheaves $\bC_p$.

However, the resulting {\it Mumford slope}:
$$\mu(E) := \frac {d(E)}{r(E)},$$
well-defined away from coherent sheaves on $\bP^2$ of finite length,  has a (weak) Harder-Narasimhan property for coherent sheaves on $\bP^2$:
$$E_0 \subset E_1 \subset \dots \subset E_n = E$$
where $E_0 := \tors(E)$ is the torsion subsheaf of $E$, and for $i > 0$, the subquotients 
$F_i := E_i/E_{i-1}$ are 
Mumford semi-stable torsion-free sheaves of strictly decreasing slopes $\mu_i := \mu(F_i)$.
\end{remark}

The following formal definition is useful:

\begin{definition}
For $s\in \bR$ and $E \in K(\cD^b(\operatorname{coh}(\bP^2)))$ define:
$$\ch(E(-s)) := \ch(E)\cdot e^{-sL},$$
where $\ch(E(-s))$ is the Chern character of $E(-sL)$ when $s\in \bZ$. 
\end{definition}

We then have the following:

\medskip

\nt {\bf Bogomolov Inequality} (see \cite[Chapter 9]{friedman}). If $E$ is a Mumford semi-stable torsion-free sheaf on $\bP^2$, then:
$$\ch_1(E(-s))\cdot L = 0 \ \Rightarrow \ \ch_2(E(-s)) \le 0$$ 

\begin{remark}
 There is an even stronger inequality obtained from:
$$\chi(\bP^2, E\otimes E^*) \le 1$$
for all stable vector bundles $E$ on $\bP^2$, but we will not need this.
\end{remark}

\begin{definition}\label{torsion-pair}
 Given $s\in \bR$, define full subcategories $\cQ_s$ and $\cF_s$ 
of $\coh(\bP^2)$ by the following conditions on their objects:

\begin{itemize}
\item $Q \in \cQ_s$ if $Q$ is torsion or if each $\mu_i > s$ in the Harder-Narasimhan filtration of $Q$.
\item $F \in \cF_s$ if $F$ is torsion-free, and each $\mu_i \le s$ in the Harder-Narasimhan filtration of $F$. 
\end{itemize}
\end{definition}

Each pair $(\cF_s,\cQ_s)$ of full subcategories therefore satisfies \cite[Lemma 6.1]{bridgeland:K3}:

\medskip

(a) For all $F\in \cF_s$ and $Q \in \cQ_s$, 
$$\Hom(Q,F) = 0$$

(b) Every coherent sheaf $E$ fits in a short exact sequence:
$$0 \rightarrow Q \rightarrow E \rightarrow F \rightarrow 0,$$
where $Q \in \cQ_s$, $F\in \cF_s$ and the extension class are uniquely determined
up to isomorphism.

\medskip

A pair of full subcategories $(\cF,\cQ)$ of an abelian category $\cA$ satisfying conditions (a) and (b) is called a {\em torsion pair}.  A torsion pair $(\cF, \cQ)$  
defines a $t$-structure on $\cD^b(\cA)$ \cite{tilting} with:
$$\cD^{\ge 0} = \{\mbox{complexes}\ E  \ | \ 
\ \rH^{-1}(E) \in \cF \ \mbox{and} \ \rH^i(E) = 0 \ \mbox{for} \ i < -1\}$$
$$\cD^{\le 0} = \{\mbox{complexes}\ E  \ | \ 
\ \rH^{0}(E) \in \cQ \ \mbox{and} \ \rH^i(E) = 0 \ \mbox{for} \ i > 0\}\ \ \ \ \ $$
The heart of the $t$-structure defined by a torsion pair consists of:
$$\{E \ | \ \rH^{-1}(E) \in \cF, \rH^0(E) \in \cQ, \ \mbox{and}\ \rH^i(E) = 0 \ \mbox{otherwise}\}.$$
The natural exact sequence:
$$0 \rightarrow \rH^{-1}(E)[1] \rightarrow E \rightarrow \rH^0(E) \rightarrow 0$$
for such an object of $\cD^b(\cA)$ implies that the objects of the heart are all given by pairs of objects 
$F \in \cF$ and $Q \in \cQ$ together with an extension class in $\Ext_{\cA}^2(Q,F)$ \cite{tilting}.

\begin{definition}
Let $\cA_s$ be the heart of the $t$-structure on $\cD^b(\coh(\bP^2))$ obtained from 
the torsion-pair $(\cF_s, \cQ_s)$ defined in Definition \ref{torsion-pair}. 
\end{definition}

 Let $(d,r)$ be degree and rank functions defined by:
$$r(E) := \ch_1(E(-s))\cdot L\ \ \ $$
$$d(E) := - \ch_0(E(-s))\cdot L^2$$
(the corresponding slope is the negative reciprocal of the Mumford slope of $E(-s)$.)
Then it is easy to see that $r(E) \ge 0$ for all objects of $\cA_s$. Furthermore, $d(E) \ge 0$, where the inequality is strict {\bf unless} 
$E$ is an object of the following form:
$$0 \rightarrow F[1] \rightarrow E \rightarrow T \rightarrow 0,$$
where $F$ is a semi-stable torsion-free sheaf satisfying $c_1(F(-s))\cdot L = 0$ and $T$ is a 
sheaf of finite length. By modifying the degree, one obtains a Bridgeland stability condition as described in the next theorem.

\begin{theorem}[Bridgeland \cite{bridgeland:K3}, Arcara-Bertram \cite{aaron:k3}, Bayer-Macr\`i \cite{bayer}]
 For each $s\in \bR$ and $t > 0$, the rank and 
degree functions on $\cA_s$ defined by:
\begin{itemize}
\item $r_t(E) := t\cdot \ch_1(E(-s))\cdot L$
\item  $d_t(E) := -({t^2}/2) \ch_0(E(-s))\cdot L^2 + \ch_2(E(-s))$
\end{itemize}
define stability conditions on $\cD^b(\operatorname{coh}(\bP^2))$ with slope function $\mu_{s,t} = d_t/r_t$. 
\end{theorem}

\begin{remark}
 By the characterization of the objects of $\cA_s$ that satisfy $r(E) = d(E) = 0$ and the Bogomolov inequality, it is immediate that the triples $(\cA_s; r_t,d_t)$ are  pre-stability conditions. The finiteness of Harder-Narasimhan filtrations is proved by Bayer-Macr\`i in \cite{bayer}.
\end{remark}

Fix a triple of numbers defining a Chern character
$$(r,c,d) \leadsto (r,cL,dL^2) \ \ \mbox{on $\bP^2$}$$
($r$ and $c$ are integers, and $d - c^2/2$ is an integer).

\medskip

\begin{corollary}[Abramovich-Polishchuk, \cite{abramovich}]
 For each point $(s,t)$ in the upper-half plane, the moduli stack $\cM_{\bP^2}(r,c,d)$ of semi-stable objects of $\cA_s$ with the given Chern character with respect to the slope function $\mu_{s,t}$ is of finite type and satisfies the semi-stable replacement property. In particular, the moduli stack of stable objects is separated, and if all semi-stable objects are stable, then the moduli stack is proper. 
\end{corollary}

\begin{remark}
 We will show that the coarse moduli spaces are projective.
\end{remark}

\section{Potential walls}\label{PotentialWalls} The 
{\it potential wall} associated to a pair of Chern characters:
$$(r,c,d) \ \mbox{and}\ (r',c',d') \ \mbox{on}\ \bP^2$$
is the following subset of the upper-half plane:
$$W_{(r,c,d),(r',c',d')} := \{(s,t) \ | \mu_{s,t}(r,c,d) = \mu_{s,t}(r',c',d')\}$$
where $\mu_{s,t}$ are the {\it slope functions} $\mu_{s,t} = d_t/r_t$ defined in \S \ref{Bridgeland-prelim}. Specifically:
$$\mu_{s,t}(r,c,d) =  \frac {-\frac {t^2}2 r + (d - sc + \frac{s^2}2 r)}{t(c - sr)}$$
and so the wall is given by:
$$W_{(r,c,d),(r',c',d')} = \{(s,t) \ | (s^2 + t^2)(rc' - r'c) - 2s(rd' - r'd) + 2(cd' - c'd) = 0\}$$
i.e. it is either a semicircle centered on the real axis or a vertical line 
(provided that one triple is not a scalar multiple of the other).

\medskip

Walls are significant because if $F, E$ are objects of $\cD^b(\coh(\bP^2))$ with
$$\ch(F) = (r,cL,dL^2) \ \mbox{and} \ \ch(E) = (r',c'L,d'L^2)$$
and if  $F \subset E$ in $\cA_s$ with $(s,t) \in W_{(r,c,d),(r',c',d')}$, 
then $E$ is not $(s,t)$-stable (by definition), but it {\bf may} be $(s,t)$-semistable and stable for 
nearby points on one side of the wall, but not the other. Crossing the wall would therefore 
change the set of $(s,t)$-stable objects. 

\medskip

Notice that by the explicit equation for the potential wall $W_{(r,c,d),(r',c',d')}$:

\medskip

(i) if $rc' = r'c$ (i.e. the Mumford slopes of the triples are the same), then the wall is the vertical line:
$$s =  \frac{cd' - c'd}{rd' - r'd}$$

(ii) otherwise the wall is the semicircle with center
$$\left(\frac{rd' - r'd}{rc' - r'c},0\right)$$
and radius
$$\sqrt{\left(\frac{rd' - r'd}{rc' - r'c}\right)^2 - 2\left(\frac{cd' - c'd}{rc' - r'c}\right)}. $$

\medskip

Now consider three cases for which
$$(r',c'L, d'L^2) = (r(E),c_1(E),\ch_2(E))$$ 
is chosen to be the Chern character of a sheaf $E$ on $\bP^2$, where we assume 
in addition that the triple $(r',c',d')$ is {\it primitive} (not an integer multiple of another such triple).

\medskip

\nt {\bf Case 1.} $E = \bC_x$, the skyscraper sheaf, so $(r',c',d') = (0,0,1)$ and:
$$\mu_{s,t}(\bC_x) = +\infty \ \mbox{for all $(s,t)$}$$
In this case each triple $(r,c,d)$ gives a potential vertical wall:
$s = \frac{c}{r}$.
However, we will see below that there are {\bf no} proper nonzero subobjects of $\bC_x$ in $\cA_s$. 
In other words, the skyscraper sheaves are {\it stable} for all 
values of $(s,t)$. 

\medskip

\nt {\bf Case 2.} $E$ is supported in codimension $1$, so $r' = 0$ and $c' > 0$.

\medskip

Walls of type (ii) are all semicircles with the {\bf same} center 
$$\left(\frac{\ch_2(E)}{c_1(E)}, 0\right).$$
Type (i) walls do not occur, 
since $r = 0$ would imply that $cd' = c'd$ at a wall, i.e. that one triple is a scalar multiple of the other. Thus the potential
walls are simply the family of semicircles, centered at a fixed point on the real axis. These semi-circles foliate the upper-half plane

\medskip

\nt {\bf Case 3.} $E$ is a Mumford-stable torsion-free sheaf with $c' = c_1(E) = 0$. 
The rank $r'$ is positive and by the Bogomolov inequality,  $d'  \le 0$. 

\medskip

Here there is one vertical wall, 
with $0 = rc' =r'c$, i.e. $c = 0$ and:
$$s = 0$$

When $c \ne 0$, the potential walls are two families of nested semicircles (with varying centers)\footnote{Maciocia \cite{maciocia} recently proved that Bridgeland walls for smooth projective surfaces of Picard rank one are nested}, one in each of the two quadrants, with centers:
$$\left(x,0\right) \ \mbox{and radius} \ \sqrt{x^2 + \frac{2\ch_2(E)}{r(E)}} \le |x| $$
(recall that $d' = \ch_2(E) \le 0$ is fixed) and 
$$x = \frac{rd' - r'd}{-r'c} = \frac{r\ch_2(E) - r(E)d}{-r(E)c}$$
Notice  that $x\in \bQ$ whenever $(r,c,d)$ are all rational. 

\medskip

We will only be interested in walls lying in the second quadrant $(x < 0)$ because a Mumford-stable torsion-free sheaf
$E$ of degree $0$ only belongs to the category $\cA_s$ if $s <0$.

\medskip

\begin{remark}
 There is little loss of generality in assuming $c_1(E) = 0$ in Case 3. Indeed, if $E$ is a 
Mumford-stable torsion-free sheaf of arbitrary rank $r'$ and degree $c'$, then the set of potential walls  
also consists of a single vertical line at $s = \frac{c'}{r'}$ and nested semicircles on either side. This can be seen most
simply by formally replacing the Chern classes $\ch(E)$ by $\ch(E(-\frac{c'}{r'}))$ (which shifts all walls by $\frac{c'}{r'}$) 
and reducing to Case 3. 
\end{remark}

Proofs of the following proposition already exist in the literature but we reprove it here in 
order to develop some techniques that will be useful later.

\begin{proposition}\label{Stable}\

\medskip

(a) Skyscraper sheaves $\bC_x$ are stable objects for all  (s,t).
 
 \medskip
 
(b) If $E \in \cA_s$ is stable for fixed $s$ and $t >> 0$, then $E \in \cQ_s$ or, if $H^{-1}(E) \ne 0$, then  
$\mH^0(E)$ is a sheaf of finite length.
 
 \medskip
 
(c) Torsion-free sheaves in $\cQ_s$  that are not Mumford semistable are not (s,t)-semistable for large $t$.
 
 \medskip

(d) Line bundles $\cO_{\bP^2}(k)$, $k > s$, are stable objects of $\cA_s$ for all $t$.

\medskip

\end{proposition}

\begin{proof} A short exact sequence $0 \rightarrow A \rightarrow E \rightarrow B \rightarrow 0$ of objects of $\cA_s$ gives 
rise to a long exact sequence of coherent sheaves:
$$0 \rightarrow \mH^{-1}(A) \rightarrow \mH^{-1}(E) \rightarrow \mH^{-1}(B) \rightarrow 
\mH^{0}(A) \rightarrow \mH^{0}(E) \rightarrow \mH^{0}(B) \rightarrow 0$$
with $\mH^{-1}(*) \in \cF_s$ and $\mH^0(*) \in \cQ_s$. In particular, if $E = \mH^0(E) \in \cQ_s$ is a coherent sheaf (i.e., $\mH^{-1}(E)=0$), 
then $A = \mH^0(A)$ is also a coherent sheaf in $\cQ_s$.

\medskip

Let $E = \bC_x$. Then $\mu_{s,t}(E) = +\infty$ for any $(s,t)$. Thus $E$ is either
$(s,t)$-semi-stable or stable, and if it is not stable, then there is a short exact sequence as above with 
$B \ne 0$,
$\mu_{s,t}(A) = +\infty$ and $A \in \cQ_s$. But the only such sheaves $A$ are torsion, supported in dimension zero. 
Thus the long exact sequence in cohomology gives:
$$0 \rightarrow \mH^{-1}(B) \rightarrow A \rightarrow \bC_x \rightarrow \mH^0(B) \rightarrow 0$$
and then $\mH^{-1}(B)$ is also torsion, violating $\mH^{-1}(B) \in \cF_s$.

\medskip

Next, suppose $E \in \cA_s$ and both $\mH^{-1}(E) \ \mbox{and}\ \mH^0(E)$ are nonzero. The cohomology sheaves
form a short exact sequence of objects of $\cA_s$:
$$0 \rightarrow \mH^{-1}(E)[1] \rightarrow E \rightarrow \mH^0(E) \rightarrow 0$$
As $t \rightarrow \infty$:

\medskip

$\mu_{s,t}(\mH^{-1}(E)[1]) \rightarrow +\infty$, whereas

\medskip

$\mu_{s,t}(\mH^{0}(E)) \rightarrow -\infty$ if $r(\mH^0(E)) > 0$,

\medskip

$\mu_{s,t}(\mH^{0}(E)) \rightarrow 0$ if $r(\mH^0(E)) = 0$ and $c_1(\mH^0(E)) > 0$, and

\medskip

$\mu_{s,t}(\mH^{0}(E)) \equiv +\infty$ if $r = c_1 = 0$. 

\medskip

Thus for large values of $t$, the inequality 
$\mu_{s,t}(\mH^{-1}(E)[1]) > \mu_{s,t}(\mH^0(E))$ would destabilize $E$ 
unless $\mH^{-1}(E)[1]$ is the zero sheaf or else $\mH^0(E)$ satisfies $r = c_1 = 0$ . This gives (b).

\medskip

Similar considerations give (c). If $E \in \cQ_s$ is not Mumford semistable, consider a sequence of sheaves:
$$0 \rightarrow A \rightarrow E \rightarrow B \rightarrow 0$$
with Mumford slopes $\mu(A) > \mu(E) > \mu(B)$.

\medskip

The slope  of $B$ satisfies $\mu(B) > s$ (otherwise $E \not \in \cQ_s$) and 
as above, the limiting $\mu_{s,t}$ slopes (to first order in $t$) for large $t$ show that $E$ is not $(s,t)$-semistable for large $t$. One can even refine the argument to show that if $E$ is not Gieseker-semistable 
and either $E$ is torsion or else $E$ is Mumford-semistable, then the higher
order term (in $t$) exhibits the instability of $E$ for large $t$. Thus $(s,t)$-stability of sheaves in $\cQ_s$ for large $t$ is equivalent to Gieseker stability.

\medskip

Let $s_0 < 0$, and suppose that $\cO_{\bP^2}$ is not $(s_0,t_0)$-semistable. Let $A \subset \cO_{\bP^2}$ 
have minimal rank among all subobjects in $\cA_{s_0}$ satisfying $\mu_{s_0,t_0}(A) > \mu_{s_0,t_0}(\cO_{\bP^2})$. Consider the (unique!) potential wall $W := W_{((r,c,d),(1,0,0))}(s_0,t_0)$
passing through $(s_0,t_0)$. By Case 3,
$W$ is a semicircle centered on the 
$x$-axis in the second quadrant, passing through the origin.
It follows that $\mu_{s,t}(A) > \mu_{s,t}(\cO_{\bP^2})$
at {\bf all} points $(s,t) \in W$. If $A \in \cQ_s$ for all $s < 0$, then each term in the 
Harder-Narasimhan filtration of $A$ with respect to the Mumford slope has slope $\ge 0$. In particular, $\ch_1(A) \ge 0$, and the 
kernel of the map $A \rightarrow \cO_{\bP^2}$ (which must map onto a sheaf with $\ch_1 = 0$) also 
satisfies $\ch_1(\mH^{-1}(B)) \ge 0$. But this contradicts the fact that $\mH^{-1}(B) \in \cF_{s_0}$ with 
$s_0 < 0$. 

\medskip

On the other hand, if $A\not\in \cQ_s$ 
for some $s < 0$, let $s_0 < s' < 0$ be the smallest such $s$. Then by assumption there is a quotient sheaf $A'$ of $A$ with $\mu(A') = s'$, and 
it follows that $\lim_{(s,t) \rightarrow (s',t')} \mu_{s,t}(A') = -\infty$ (taken along the wall $W$). The kernel sheaf $A'' \subset A$ of the map to $A'$ is then nonzero, the map $A'' \rightarrow \cO_{\bP^2}$ is also nonzero and determines a destabilizing subobject of $\cO_{\bP^2}$ at  $(s_0,t_0)$, contradicting our assumption on the minimality of the rank of $A$ among all destabilizing subobjects. This proves that $\cO_{\bP^2}$ is stable
at an arbitrary point $(s_0,t_0)$, and
since tensoring by $\cO_{\bP^2}(k)$ translates the upper half plane, it follows
that all line bundles are stable for all values of $(s,t)$ with $s< k$. This gives (d). 
\end{proof}

\medskip

A potential wall $W_{(r,c,d),\ch(E)}$ for $E$ will be an {\it actual} wall if the equality of $(s,t)$-slopes is realized 
by a semi-stable subobject $A \subset E$ with $\ch(A) = (r,c,d)$ at {\it some} point of the wall. When $E$ is a coherent sheaf, the rank of a 
subsheaf of $E$ is bounded by the rank of $E$. Unfortunately, while a subobject $A \subset E$ in one of the 
categories $\cA_s$ is necessarily also a sheaf, it may not be a subsheaf and indeed it may, a priori, be a sheaf of 
arbitrarily large rank. The following lemma will be useful for bounding such subobjects.  

\medskip

\begin{lemma}\label{every-point-wall}

Let $E$ be a coherent sheaf on $\bP^2$ (not necessarily Mumford-semistable) of positive rank
with $\ch_1(E) = 0$ satisfying:
$$\ch_2(E) < 0$$
and suppose $A \rightarrow E$ is a map of coherent sheaves which is an inclusion of 
$\mu_{s_0,t_0}$-semi-stable objects of $\cA_{s_0}$ of the same slope for some
$$(s_0,t_0) \in W := W_{(\ch(A),\ch(E))}$$

\medskip

\nt Then $A \rightarrow E$ is an inclusion of $\mu_{s,t}$-semi-stable objects of $\cA_s$ of the same slope
for {\bf every} point $(s,t) \in W$. 

\end{lemma}

\begin{proof} First, we prove that $E \in \cQ_s$ for all $(s,t) \in W$. Since $E \in \cQ_{s_0}$ by assumption, it follows
that $E \in \cQ_s$ for all $s \le s_0$. The set of $(s,t) \in W$ such that $E \not\in \cQ_{s}$, if non-empty, has a well-defined infimum $s' > s_0$, and then by definition, the end of the Harder-Narasimhan filtration for $E$ must be:
$$0 \rightarrow E'' \rightarrow E \rightarrow E' \rightarrow 0$$
with $\mu(E') = s'$. But then $\lim_{(s,t) \rightarrow (s',t')^-} \mu_{s,t}(E') = -\infty$ and so in particular, 
$\mu_{s,t}(E'') > \mu_{s,t}(E)$ for $(s,t)$ near $(s',t')$ on the wall. But the walls for $E$ are disjoint (as in Case 3 above), and it follows that 
$\mu_{s,t}(E'') > \mu_{s,t}(E)$ for {\bf all} $(s,t)$ on $W$, including $(s_0,t_0)$, contradicting the assumption that $E$ is 
$(s_0,t_0)$-semistable. 
Similarly, we may conclude that $A \in \cQ_s$ for all $(s,t) \in W$, since otherwise $A$ would admit a subsheaf 
$A''$ that destabilizes $E$ at $(s',t')$ and hence also at $(s_0,t_0)$. 

\medskip

Next, consider the exact sequence of cohomology sheaves:
$$0 \rightarrow \rH^{-1}(B) \rightarrow A \rightarrow E \rightarrow \rH^0(B) \rightarrow 0$$
for $B = E/A$ in $\cA_{s_0}$. It is immediate that $\rH^0(B) \in \cQ_s$ for all $(s,t) \in W$, since the quotient of a 
sheaf in $\cQ_s$ is also in $\cQ_s$. The issue is to show that $\rH^{-1}(B) \in \cF_s$ for all $s < s_0$ and $(s,t) \in W$. But this follows the same argument as the previous paragraph. If not, then there is a Harder-Narasimhan filtration starting with:
$$0 \rightarrow F'' \rightarrow \rH^{-1}(B) \rightarrow F' \rightarrow 0$$
with $F''$ Mumford semistable, $\mu(F'') = s''$ and $(s'',t'') \in W$. Then:
$$0 \rightarrow F''[1] \rightarrow B \rightarrow B' \rightarrow 0$$
has $\mu_{s'',t''}(B') < \mu_{s'',t''}(B) = \mu_{s'',t''}(E)$, which violates the assumption that $E$ is 
$\mu_{s_0,t_0}$-semistable. 
\end{proof}

The following corollary is immediate:

\begin{corollary}\label{main-cor}
Suppose $E$ is a coherent sheaf as in Lemma \ref{every-point-wall}, and let $(s_1,0)$ and $(s_2,0)$, with $s_1<s_2$, be the two 
intersection points of the (semi-circular) wall $W$ with the $x$-axis. Let $K = \ker(A \rightarrow E)$
be the kernel sheaf. Then:
$$K \in \cF_{s} \ \mbox{and}\ A \in \cQ_{s}$$ for every $s_1< s< s_2$.
\end{corollary}

The next corollary completes the circle of ideas from Proposition \ref{Stable}:

\medskip

\begin{corollary}\label{stable-Mum}
Mumford-stable torsion-free coherent sheaves $E \in \cQ_s$ are $(s,t)$-stable objects of $\cA_s$ for $t >> 0$.  
\end{corollary}

\begin{proof} When $E = \cO_{\bP^2}(k)$ is a line bundle, the Corollary follows from Proposition \ref{Stable} (d) . When $E$ has higher rank, then the 
Bogomolov inequality is sharp, and moreover, by (formally) twisting by $-\mu(E)$, we may as well 
assume that $\ch_1(E) = 0$ and $\ch_2(E) < 0$.  Observe that finiteness of the Harder-Narasimhan filtration implies that given any Mumford-stable torsion-free sheaf $E \in \cQ_s$, we can find a $t$ such that $E$ is $(s,t)$ Bridgeland stable. Here we show that $t$ can be chosen uniformly depending only on the invariants of $E$.
\medskip

Suppose $A \rightarrow E$ satisfies the conditions of Lemma \ref{every-point-wall} for some value $(s,t) \in W$. Our strategy is to find a uniform upper bound on the diameter of the (semi-circular) wall $W$. It will follow that above such a wall, $E$ must be $(s,t)$-stable. 

\medskip

We separate into two cases:

\medskip

(i) $A \subset E$ is a subsheaf, necessarily satisfying $\mu(A) < \mu(E) = 0$. In fact, we can say more, namely 
that $\mu(A) \le - 1/\ch_0(E)$. But then by Lemma \ref{every-point-wall}, it follows that no wall $W_{\ch(A),\ch(E)}$ that exhibits $E$ as a semi-stable object may extend past $s = - 1/\ch_0(E)$ on the $x$-axis. It follows that for {\bf all} values of $(s,t)$ above the wall $W_{(*,\ch(E))}$ passing through $(- 1/\ch_0(E),0)$, the stable vector bundle $E$ may {\bf not} be destabilized by a subsheaf of $E$. This is because for $t >> 0$, any given subsheaf $A \subset E$ of smaller Mumford slope has smaller $(s,t)$-slope. Thus if $A \subset E$ destabilized $E$ for some value $(s,t)$, then there would necessarily be a wall above $(s,t)$ for which Lemma \ref{every-point-wall} applies. 

\medskip

(ii) $A \rightarrow E$ has nontrivial kernel $K \subset A$. Let $(s_1,0)$ and $(s_2,0)$ be the intersections 
of $W_{A,E}$ with the $x$-axis. Then by Corollary 3.2,
$$\mu(K) \le s_1 \ \mbox{and}\ \mu(A) > s_2.$$

Now suppose $\mu(K) \le -(r(E)+1)$. Since $E$ is Mumford stable, all quotients of $E$ have non-negative $c_1$. Hence, $c_1(A) \le c_1(K)$. Combining this with:
$$ \ch_0(A) \le \ch_0(K) + \ch_0(E)$$
it follows that:
$$\mu(A) = \frac{c_1(A)}{\ch_0(A)} \le \frac{c_1(K)}{\ch_0(A)} \le \frac{\ch_0(K)\mu(K)}{\ch_0(K) + \ch_0(E)} \le -1$$

This means that {\bf any} semicircular wall for such an $A$ must be bounded by the larger of the 
wall through $(-(r+1),0)$ and the wall through $(-1,0)$. This gives the desired uniform bound.
\end{proof}

\begin{remark}
 In specific examples, one can do much better than these bounds, as we shall see when we make our detailed analysis of the Hilbert scheme.
\end{remark}

\begin{remark}
Proposition \ref{Stable} and Corollary \ref{stable-Mum} allow us to identify  the Hilbert scheme $\PP^{2[n]}$ with the coarse moduli schemes of  $\cM_{s,t}(1,0,-n)$  for $s<0$ and $t$ sufficiently large. 
\end{remark}

\medskip

\section{The quiver region}\label{s-quiver} Fix an integer $k \in \bZ$ and consider the three objects:
$$\cO_{\bP^2}(k-2)[2], \ \cO_{\bP^2}(k-1)[1], \ \cO_{\bP^2}(k) \in \cD^b(\coh(\bP^2))$$
This is an ``Ext-exceptional'' collection, in the sense of \cite[Definition 3.10]{MacriStability}, where it is shown that 
the extension-closure of these three objects:
$$\cA(k) := \langle \cO_{\bP^2}(k-2)[2], \ \cO_{\bP^2}(k-1)[1], \ \cO_{\bP^2}(k)\rangle$$
is the heart of a $t$-structure. Moreover, Macr\'i explains that:

\medskip

\begin{lemma}[Macr\'i] \cite[Lemma 3.16]{MacriStability}\label{quiver} If $\cA$ is the heart of a $t$-structure and 
$$\cO_{\bP^2}(k-2)[2], \ \cO_{\bP^2}(k-1)[1], \ \cO_{\bP^2}(k) \in \cA$$
then $\cA = \cA(k)$.
\end{lemma}

\medskip

The objects of $\cA(k)$ are complexes:
$$\bC^{n_0} \otimes _{\bC} \cO_{\bP^2}(k-2) \rightarrow \bC^{n_1} \otimes _{\bC} \cO_{\bP^2}(k-1)\rightarrow 
\bC^{n_2} \otimes _{\bC} \cO_{\bP^2}(k).$$
In particular, a subobject 
in $\cA(k)$ of an object $E^\bullet$ of dimensions $(n_0,n_1,n_2)$ has dimensions $(m_0,m_1,m_2)$ with 
$m_i \le n_i$ for each $i$. Thus, quite unlike the category of coherent sheaves (or any of the categories $\cA_s$ above) 
there are only finitely many possible invariants for subobjects of an object with given invariants. It also immediately follows that:

\begin{observation}
 For {\it any} choices of $\zeta_0,\zeta_1,\zeta_2 \in \bH$, if we define:
$$Z_{\zeta_0,\zeta_1,\zeta_2}(\cO_{\bP^2}(k-i)[i]) := \zeta_i$$
then the pair $(\cA(k), Z_{\zeta_0,\zeta_1,\zeta_2})$ is a {\it stability} condition. Here $Z= d+ir$ in terms of the rank and degree functions defined before.
\end{observation}

\begin{remark}
Here the finiteness of Harder-Narasimhan filtrations is trivial. Also, notice that this space of stability conditions is 
three (complex) dimensional, with fixed $t$-structure, unlike the upper-half plane, which was two real 
dimensional, with varying $t$-structures. Clearly $\cA(k) \ne \cA_s$ for any $k$ or $s$ 
(there is no coherent sheaf shifted by $2$ in any $\cA_s$). Nevertheless, we will find isomorphisms 
of moduli spaces of stable objects. 
\end{remark}

For fixed $k$, the conversion $(n_0, n_1, n_2) \mapsto (r,c,d)$ from dimensions to Chern classes is:
\medskip
$$C := \left[ \begin{array}{ccc} 1 & -1 & 1 \\ \\
k-2 & -(k-1) & k \\ \\
{\frac{(k-2)^2}2} & \frac{-(k-1)^2}2 & \frac{k^2}2 \\ \end{array}\right]$$

\medskip

Vice versa, the conversion from Chern classes to dimensions is:

\medskip

$$C^{-1} = \left[ \begin{array}{ccc}   \frac{k(k-1)}2 & \frac{-(2k-1)}2 & 1 \\ \\
k(k-2) & -(2k-2) &  2 \\ \\ 
\frac{(k-1)(k-2)}2& \frac{-(2k-3)}2 & 1 \\ \end{array}\right]$$

\bigskip

\begin{example} (i) For each integer $k$, the twisted Koszul complex:

$$\bC^{1} \otimes _{\bC} \cO_{\bP^2}(k-2) \rightarrow \bC^{2} \otimes _{\bC} \cO_{\bP^2}(k-1)\rightarrow 
\bC^{1} \otimes _{\bC} \cO_{\bP^2}(k)$$
is exact except at the right, where the cokernel is isomorphic to the skyscraper sheaf $\bC_x$. This matches the dimension 
computation:
$$C^{-1}\left[\begin{array}{c} 0 \\ 0 \\ 1 \end{array}\right] = \left[\begin{array}{c} 1 \\ 2 \\ 1\end{array} \right] $$

\medskip

(ii) Every stable torsion-free sheaf  $E$ of degree slope: 
$$-1 < \mu(E) \le 0$$
is the {\it middle} cohomology 
of a ``monad'' that is exact elsewhere:

\medskip

$$\bC^{n_0} \otimes _{\bC} \cO_{\bP^2}(-1) \rightarrow \bC^{n_1} \otimes _{\bC} \cO_{\bP^2}\rightarrow 
\bC^{n_2} \otimes _{\bC} \cO_{\bP^2}(1)$$

\medskip

\nt or in other words, $E[1] \in \cA(1)$ (see, e.g. \cite{BarthMonads}).
\end{example}

\medskip

In the case we will consider, $\cI_Z$ is the ideal sheaf of a subscheme $Z \subset \bP^2$ 
of length $l(Z) = n$. This is 
stable and torsion-free with Chern character $(1,0,-n)$, so as an object of $\cA(1)$, the associated monad for $\cI_Z[1]$ 
has Chern character $(-1,0,n)$ and dimension invariants:
$$(n_0,n_1,n_2) = (n, 2n+1, n).$$

Other than the monad, we will assume $k = -d$ is {\bf non-positive}.

\medskip

\nt {\bf The Dimension Invariants} for $\cI_Z[1]$ in $\cA(-d)$ are:
$$C^{-1}\left[ \begin{array}{r} -1 \\ 0 \\ n \end{array}\right] = \left(n -  \frac{d(d+1)}2, \ 2n - d(d+2), \ n - \frac{(d+1)(d+2)}2\right)$$

In particular, 
\begin{equation}
n \ge  \frac{(d+1)(d+2)}2
\end{equation}
is a {\it necessary} condition for {\bf any} object with Chern classes $(-1,0,n)$ to belong to $\cA(-d)$. 
On the other hand, 
\begin{equation}
n \le \frac{d(d+1)}2
\end{equation}
is needed for an object with Chern character $(1,0,-n)$ to be in $\cA(-d)$. 

\medskip

Suppose $(\cA,Z)$ is an arbitrary stability condition on $\cD^b(\coh(\bP^2))$. 
For each integer $i \in \bZ$, there is a stability condition $(\cA[i], Z[i])$ with:
$$\cA[i] := \{ A[i] \ | \ A \ \mbox{is an object of $\cA$}\} \ \ \mbox{and}\ \ Z[i](A[i]) := (-1)^iZ(A[i]).$$

More interestingly, one can {\it interpolate} between these integer shifts. For each $0 < \phi < 1$, define:

\medskip

$\bullet$ $\cQ_\phi = \langle Q \in \cA \ | \ Q \ \mbox{is stable with} \  \arg(Z(Q)) > \phi \pi \rangle$.

\medskip

$\bullet$ $\cF_\phi = \langle F \in \cA \ | \ F \ \mbox{is stable with} \  \arg(Z(F)) \le \phi \pi \rangle$.

\medskip

\nt and define $\cA[\phi] = \langle \cQ_\phi, \cF_\phi[1]\rangle$ and $Z[\phi](E) = e^{-i\pi\phi}Z(E).$

\medskip

This extends to an action of $\bR$ on the manifold of stability conditions. 

\medskip

\nt {\it Remark.} This action of $\bR$ is the restriction of an action by $\widetilde{\GL(2,\bR)^+}$ (the universal cover
of the set of matrices of positive determinant) established by Bridgeland in \cite{bridgeland:stable} but we will not need this larger group action.
It is important to notice that moduli spaces of stable objects are {\bf unaffected} by the action. Specifically:

\medskip

(a) Stable objects of $\cA$ with $\arg(Z)/\pi = \psi > \phi$ are identified with
stable objects of $\cA[\phi]$ with  $\arg(Z[\phi])/\pi = \psi - \phi$.

\medskip

(b) Stable objects of $\cA$ with $\arg(Z)/\pi = \psi \le \phi$ are identified with stable objects of 
$\cA[\phi]$ with $\arg(Z[\phi])/\pi =  1 + (\psi - \phi)$ via $A \mapsto A[1]$. 

\medskip

Finally, we have the following:

\medskip

\nt \begin{proposition}\label{QuiverRegions} If $(s,t)$ satisfy:
\begin{equation}
(s - (k-1))^2 + t^2 < 1
\end{equation}
then each moduli space of $(s,t)$-stable objects (with fixed invariants) is isomorphic to a moduli space
of stable objects in $\cA(k)$ for suitable choices of 
$\zeta_0,\zeta_1,\zeta_2$ (depending upon $(s,t)$). 
\end{proposition}

\begin{proof} First note that $\cO_{\bP^2}(k)$ and  $\cO_{\bP^2}(k-2)[1]$ are stable objects of  $\cA_s$ for all $k-2 < s < k$
(see \cite{aaron:k3} for the latter).  
Moreover, the semicircle $(s-(k-1))^2 + t^2 = 1$ is the potential wall corresponding to:
$$\arg(Z_{(s,t)}(\cO_{\bP^2}(k-2)[1])) = \arg(Z_{(s,t)}(\cO_{\bP^2}(k)))$$
Below this wall, the former has smaller arg than the latter. 

\medskip

It is useful to divide the region (3) into four subregions.

\medskip

(R1) is the region: $s \ge k - 1,  (s - (k-\frac 12))^2 + t^2 > 1$

\medskip

(R2) is the region: $s > k - 1,  (s - (k-\frac 12))^2 + t^2 \le 1$

\medskip

(R3) is the region: $s < k - 1,  (s - (k-\frac 32))^2 + t^2 > 1$

\medskip

(R4) is the region: $s < k - 1,  (s - (k-\frac 32))^2 + t^2 \le 1$

\medskip

Within these regions, we have the following inequalities on the args
(suppressing the subscript on the $Z$).

\medskip

(R1) $\cO_{\bP^2}(k-2)$ and $\cO_{\bP^2}(k-1)$ both shift by $1$, and:
$$\arg(Z(\cO_{\bP^2}(k-2)[1])) < \arg(Z(\cO_{\bP^2}(k))) <  \arg(Z(\cO_{\bP^2}(k-1)[1] ))$$

(R2) $\cO_{\bP^2}(k-2)$ and $\cO_{\bP^2}(k-1)$ both shift by $1$, and:
$$\arg(Z(\cO_{\bP^2}(k-2)[1])) <  \arg(Z(\cO_{\bP^2}(k-1)[1])) \le \arg(Z(\cO_{\bP^2}(k)))$$

(R3) $\cO_{\bP^2}(k-2)$ shifts by $1$, and:
$$\arg(Z(\cO_{\bP^2}(k-1))) < \arg(Z(\cO_{\bP^2}(k-2)[1])) <  \arg(Z(\cO_{\bP^2}(k)))$$

(R4) $\cO_{\bP^2}(k-2)$ shifts by $1$, and:
$$\arg(Z(\cO_{\bP^2}(k-2)[1])) \le  \arg(Z(\cO_{\bP^2}(k-1))) < \arg(Z(\cO_{\bP^2}(k)))$$

For $(s,t)$ in each region, there is a choice of $\phi(s,t) \in (0,1)$ so that:
$$\cO_{\bP^2}(k-2)[2], \cO_{\bP^2}(k-1)[1], \cO_{\bP^2}(k) \in \cA[\phi(s,t)]$$
By Lemma \ref{quiver}, it follows that $\cA[\phi(s,t)] = \cA(k)$ and $(s,t)$-stability is the same as 
$Z[\phi(s,t)]$-stability for $\zeta_i := Z_{(s,t)}[\phi(s,t)](\cO_{\bP^2}(k-i)[i])$. Thus the two stability conditions
are in the same $\bR$-orbit, and their moduli spaces are isomorphic.
\end{proof}

\medskip

\begin{corollary}\label{AllQuiver}
For {\it every} $(s_0,t_0)$ and every choice of invariants, the moduli space of 
$(s_0,t_0)$-stable objects is isomorphic to a moduli space of stable objects 
of $\cA(k)$ for some choice of $k$ and $\zeta_0,\zeta_1,\zeta_2$.

\end{corollary}

\begin{proof} The moduli spaces of stable objects of fixed invariants $(r',c',d')$ 
remain unchanged as $(s,t)$ moves along the unique potential wall $W_{(*,(r',c',d'))}(s,t)$ through $(s,t)$. Since each such wall 
is either a semicircle or vertical line, it will intersect one (or more) of the quiver regions treated in Proposition \ref{QuiverRegions}, and 
then by that Proposition, the moduli space is isomorphic to a moduli space of the desired type.  
\end{proof}

\begin{corollary}\label{FiniteWalls} Let $E$ be a Mumford-stable torsion-free sheaf on $\bP^2$ with 
primitive invariants. 
There are finitely many isomorphism types of moduli spaces of $(s,t)$-stable 
objects with invariants $\ch(E)$. 

\end{corollary}

\begin{proof} As usual, we will assume for simplicity that $c_1(E) = 0$. 
The set of potential walls $W_{((r,c,d),\ch(E))}$ where the isomorphism type of the moduli spaces 
{\it might} change, due to a subobject with invariants $(r,c,d)$, 
consists of a nested family of semicircles, all of which contain $(-\sqrt{|\frac{2\ch_2(E)}{r(E)}|},0)$ 
in the interior, by the analysis of \S \ref{PotentialWalls}. Each potential wall intersects a ``quiver region'' for some 
$-\sqrt{|\frac{2\ch_2(E)}{r(E)}|} < k \le 0$ where the moduli spaces of stable objects are identified with moduli spaces of 
stable objects in the categories $\cA(k)$ by Proposition \ref{QuiverRegions}. But in the latter categories, there are only finitely many 
possible invariants of an object with any given dimension invariants $(n_0,n_1,n_2)$. From this it 
follows immediately that for each $k$ only finitely many of the potential walls actually yield semi-stable objects with invariants $\ch(E)$. 
Since there are only finitely many $k$ to consider, the corollary follows.
\end{proof}

\section{Moduli of stable objects are GIT quotients}\label{s-quiver-moduli}
 By Corollary \ref{AllQuiver}, in order to prove the projectivity of 
the moduli spaces of $(s,t)$-stable objects with fixed Chern classes, we only need to establish:

\medskip

\begin{proposition}\label{KingsModuli} The moduli spaces of stable objects of  (non-negative!) dimension invariants 
$\vec n = (n_0,n_1,n_2)$, for each stability condition $Z = (\zeta_0,\zeta_1,\zeta_2)$ on $\cA(k)$, 
may be constructed by Geometric Invariant Theory. In particular, these moduli spaces are quasi-projective, and projective 
when equivalence classes of semi-stable objects are included.

\end{proposition} 

\begin{proof} We reduce this to Proposition 3.1 of A. King's paper \cite{king}, on moduli of 
quiver representations. First, we may assume without loss of generality that:
$$\arg(Z(n_0,n_1,n_2)) < \pi$$ 
since otherwise $n_i \ne 0 \Rightarrow \arg(\zeta_i) = \pi$, and there is either a single stable object (one of the 
generators of $\cA(k)$), or
else $n_0 + n_1 + n_2 > 1$ and  
there are no $Z$-stable objects with these invariants. 

\medskip

Let $\zeta_i = (a_i, b_i)$ so that, given dimension invariants $\vec d =  (d_0,d_1,d_2)$, 
$$Z(\vec d) = (\vec d \cdot \vec a, \vec d \cdot \vec b).$$
The criterion for stability will not change if we substitute:
$$\vec a \leftrightarrow \vec a -\vec b\left( \frac{\vec n \cdot \vec a}{\vec n \cdot \vec b}\right)$$
and so we may assume without loss of generality, that:
$$\vec n \cdot \vec a = \mbox{Re}(Z(n_0,n_1,n_2)) = 0$$
and an object $E$ with invariants $\vec n$ is stable if and only if:
$$d_0a_0 + d_1a_1 + d_2a_2 < 0$$
for all dimension invariants $\vec d$ of subobjects $F \subset E$. This is invariant under scaling $\vec a$, 
and since there are only finitely many $\vec d$ to check, we may assume that $\vec a \in \bZ^3$. Thus, our 
objects are complexes:
$$\bC^{n_0} \otimes \cO_{\bP^2}(k-2) \rightarrow \bC^{n_1}\otimes \cO_{\bP^2}(k-1) \rightarrow \bC^{n_2}\otimes \cO_{\bP^2}(k)$$
and our stability condition reduces to a triple $(a_0,a_1,a_2)$ of integers satisfying $n_0a_0 + n_1a_1 + n_2a_2 = 0$, with respect to 
which the complex is stable if and only if every sub-complex with invariants $(d_0,d_1,d_2)$ satisfies $d_0a_0 + d_1a_1 + d_2a_2 < 0$. 

\medskip

These complexes are determined by two triples of matrices:
$$M_x, M_y, M_z: \bC^{n_0} \rightarrow \bC^{n_1} \ \mbox{and}\ N_x, N_y, N_z: \bC^{n_1} \rightarrow \bC^{n_2}$$
satisfying $N_*M_* = 0$, and in particular, they are parametrized by a closed subscheme of the affine space of all 
pairs of triples of matrices (= representations of the $\bP^2$-quiver). In this context, King \cite{king} constructs the
the geometric-invariant-theory quotient by the action of $\GL(n_0)\times \GL(n_1)\times \GL(n_2)$ with the property 
that the quotient parametrizes $(a_0,a_1,a_2)$-stable (or equivalence classes of semi-stable) quiver representations. Our moduli space of 
complexes is, therefore, the induced quotient on the invariant subscheme cut out by setting the compositions of the matrices to zero. 
\end{proof}

\begin{remark}
There is a natural line bundle on the moduli stack of complexes, 
defined as follows \cite{king}. A family of complexes on $\bP^2$ parametrized by a scheme $S$ is a complex:
$$U(k-2) \rightarrow V(k-1) \rightarrow W(k) \ \mbox{on}\ S \times \bP^2$$
where $U,V,W$ are vector bundles of ranks $n_0,n_1,n_2$ pulled back from $S$, twisted, respectively,  by the pullbacks of 
$\cO_{\bP^2}(k-2), \cO_{\bP^2}(k-1), \cO_{\bP^2}(k)$. 

In this setting, the ``determinant'' line bundle on $S$:
$$(\wedge^{n_0}U)^{\otimes a_0} \otimes (\wedge^{n_1}V)^{\otimes a_1} \otimes (\wedge^{n_2}W)^{\otimes a_2}$$
is the pull back of  the ample line bundle on the moduli stack of complexes that restricts to the ample line bundle 
on the moduli space of semi-stable complexes determined by Geometric Invariant Theory.
\end{remark}

\section{Walls for the Hilbert scheme}\label{sec-actual} 

Here, we explicitly describe a number of ``actual'' walls for the moduli of 
$(s,t)$-stable objects with Chern classes $(1,0,-n)$. The results of this section will later be matched with the computation of the stable base locus decomposition in \S \ref{s-examples}. There are two types of walls we will consider, for which ``crossing the wall'' means 
decreasing $t$ (with fixed $s$) across a critical semicircle.

\medskip

\begin{itemize}

\item {\bf Rank One  Walls} for which an ideal sheaf $\cI_Z$ is $(s,t)$-destabilized by a subsheaf, necessarily of rank one and of the form:
$$0 \rightarrow \cI_W(-d) \rightarrow \cI_Z \rightarrow \cI_Z/\cI_W(-d) \rightarrow 0$$
Crossing the wall replaces such ideal sheaves with rank one sheaves $\cE$ that are extensions of the form:
$$0 \rightarrow \cI_Z/\cI_W(-d) \rightarrow \cE \rightarrow \cI_W(-d) \rightarrow 0$$
Notice that $\cE$ has a non-trivial torsion subsheaf, so is obviously not a stable sheaf in the ordinary sense (i.e. for $t >> 0$).

\medskip

\item {\bf Higher Rank Walls} for which an ideal sheaf $\cI_Z$ is $(s,t)$-destabilized by a sheaf $A$ of rank $\ge 2$ which is a 
sub-object of $\cI_Z$ in the category $\cA_s$. If we denote by $F$ the quotient, which is necessarily a two-term complex, then there is an associated exact sequence of 
cohomology sheaves: 
$$0 \rightarrow H^{-1}(F)  \rightarrow A \rightarrow \cI_Z \rightarrow H^0(F) \rightarrow 0$$
and crossing this wall produces complexes $E$ whose cohomology sheaves fit into a long exact sequence of the form:
$$0 \rightarrow H^{-1}(F) \rightarrow H^{-1}(E) \rightarrow 0 \rightarrow H^0(F) \rightarrow  H^0(E) \rightarrow A \rightarrow 0$$

\medskip

\item {\bf The Collapsing Wall},  the innermost semicircle, with the property that for $(s,t)$ on the semicircle, 
there are semi-stable but no stable objects, and for $(s,t)$ in the interior of the semi-circle, there are no semi-stable objects whatsoever.

\end{itemize}

\medskip

\begin{remark}
 In this paper, we will not consider the walls for which no ideal sheaf is destabilized (as the wall is crossed). By Proposition \ref{Stable}, any such wall has to be contained in a higher rank wall. It is not 
clear whether  such walls exist. 
\end{remark}
\medskip

The potential walls for $\PP^{2[n]}$ in the $(s,t)$-plane with $s<0$ and $t>0$ are semi-circles with center $(x,0)$ and radius $\sqrt{x^2 - 2n}$, where $$x = \frac{ch_2(\FF) + r( \FF) n}{ c_1 (\FF)}$$ and $\FF$ is the destabilizing object giving rise to the wall.

\medskip

{\bf Rank One:} The rank one destabilizing sheaves have the form:
$$\cI_W(-k) \subset \cI_Z.$$ Observe that any such subsheaf is actually a subobject in every category $\cA_s$ with $s<-k$.
They give rise to walls $W_x$ with 
$$x = -\frac nk - \frac k2 + \frac{l(W)}{k}.$$

\begin{observation}\label{rank-1}
If the potential wall corresponding to $\OO_{\PP^2}(-k)$ is contained inside the collapsing wall, then all the potential walls arising from $\cI_W(-m)$ with $m \geq k$ and $m^2 \leq 2n$ are contained in the collapsing wall. 
\end{observation}

\begin{proof}
Let $W_{x_1}$ be the wall corresponding to $\OO_{\PP^2}(-k)$. Let $W_{x_2}$ be the wall corresponding to $\cI_W(-m)$. Since the potential walls are nested semi-circles, it suffices to show that $x_1 \leq x_2$. We have that $$x_1 = -\frac nk - \frac k2, \ \ x_2 = - \frac nm - \frac m2 + \frac{l(W)}{m}.$$ If for contradiction we assume that $x_1 > x_2$, we obtain the inequality $$ -\frac nk - \frac k2 > - \frac nm - \frac m2 + \frac{l(W)}{m},$$ which implies the inequality $$km (m-k) > 2n (m-k) + 2k\  l(W).$$ This contradicts the inequality $$km \leq m^2 \leq 2n$$ of the hypotheses.
\end{proof}

\begin{remark}
If $\cI_W(-m)$ gives rise to an actual Bridgeland wall for $\PP^{2[n]}$, then by Proposition \ref{Stable} and Corollary \ref{main-cor}, $m^2 \leq 2n$. 
\end{remark}

{\bf Higher Rank Walls:} 
Suppose that $\phi: \FF \rightarrow \cI_Z$ is a sheaf destabilizing $\cI_Z$ at a point $p$ of the wall $W_x$ with center $(x, 0)$. Let $K$ be the kernel $$0 \rightarrow K \rightarrow \FF \rightarrow \cI_Z.$$
By Corollary \ref{main-cor}, both $\FF$ and $K[1]$ have to belong to all the categories $\cA_s$ along the wall $W_x$. From this, we conclude the inequalities $$x-\sqrt{x^2 - 2n} \geq \frac{d(K)}{r(K)},$$ and  $$ x + \sqrt{x^2 -2n} \leq \frac{d(\FF)}{r(\FF)}.$$

Since we have that $$d(\FF) \leq d(K)$$ and $$r(\FF) = r(K) + 1,$$ we can combine these inequalities to obtain the following set of inequalities 

$$x + \sqrt{x^2 -2n} \leq \frac{d(\FF)}{r(\FF)} \leq \frac{d(K)}{r(K)} \frac{r(K)}{r(\FF)} \leq \left(\frac{r(\FF) -1}{r(\FF)}\right) \left(x - \sqrt{x^2 - 2n}\right).$$

Rearranging the inequality, we get the following bound on the center of a Bridgeland wall 
\begin{equation}\label{main}
x^2 \leq \frac{n (2r(\FF)-1)^2}{2 r(\FF)(r(\FF) -1)}.
\end{equation}
Similarly, we get the following inequality for the degree $d(\FF)$: $$r(\FF) \left(x + \sqrt{x^2 -2n}\right) \leq d(\FF) \leq \left(r(\FF) -1\right)  \left(x - \sqrt{x^2 - 2n}\right).$$
The following observation will be useful in limiting the number of calculations we need to perform.

\begin{observation}\label{numerics}
Suppose that $s > r>1$, then $$\frac{(2s-1)^2}{2s(s-1)} < \frac{(2r-1)^2}{2r(r-1)}.$$
To see this inequality, notice that $$\frac{(2r-1)^2}{2r(r-1)} = 2 + \frac{1}{2r(r-1)}.$$ Hence, if $s > r > 1$, then $2s(s-1) > 2r(r-1)$ and the claimed inequality follows.  Thus if Inequality (\ref{main}) forces all rank $r$ walls to lie within the collapsing wall, rank $s$ walls also lie in the collapsing wall.
\end{observation}

\medskip

\begin{remark}
In particular, if  $x< - \frac{n}{2} -1$, then the only Bridgeland walls with center at $x$ correspond to rank one walls with $k=1$. Therefore, $x= -n-\frac{1}{2}+ l(W)$ and an ideal sheaf $\cI$ is destabilized by $\cI_W(-1)$ when crossing the wall with center at $x$. On the other hand, let $\frac{n}{2} \leq d \leq n-1$ be an integer. By Proposition \ref{DkBase}, the only walls in the stable base locus decomposition of the effective cone of $\PP^{2[n]}$ contained in the convex cone generated by $H$ and $H - \frac{1}{n}B$ are spanned by the rays $H - \frac{1}{2d}B$. Setting $y = -d$, we see that the transformation $x = y - \frac{3}{2}$ gives a one-to-one correspondence between these walls. Furthermore, a scheme is contained in the stable base locus after crossing the wall spanned by $H + \frac{1}{2y}B$ if and only if the corresponding ideal sheaf $\cI$ is destabilized at the Bridgeland wall with center at $x= y-\frac{3}{2}$.
\end{remark}

\section{Explicit Examples}\label{s-examples}

In this section, we work out the stable base locus decomposition and the Bridgeland walls in the stability manifold of $\PP^{2[n]}$ for $n \leq 9$. These examples are the heart of the paper and demonstrate how to calculate the wall-crossings in given examples. We preserve the notation for divisor classes and curve classes introduced in \S \ref{s-ample}. 

\smallskip

In each example, we will list the walls and let the reader check that there is a one-to-one correspondence between the Bridgeland walls with center at $x<0$ and the walls spanned by $H+ \frac{1}{2y}B$, $y<0$, in the stable base locus decomposition given by $x= y-\frac{3}{2}$.

To compactly describe the stable base locus decomposition of $\PP^{2[n]}$ for $n\leq 9$ we will need to introduce a little more notation.  
\begin{itemize}
\item Let $A_{2,k}(n)$ be the curve class in $\PP^{2[n]}$ given by fixing $k-1$ points on a conic curve, fixing $n-k$ points off the curve, and allowing an $n$th point to move along the conic.  We have $$A_{2,k}(n)\cdot H = 2 \qquad A_{2,k}(n)\cdot B = 2(k-1).$$

\item Let $L_k(n)$ be the locus of schemes of length $n$ with a linear subscheme of length at least $k$.  Observe that $L_k(n)$ is swept out by irreducible curves of class $C_k(n)$.

\item Let $Q_k(n)$ be the locus of schemes of length $n$ which have a subscheme of length $k$ contained in a conic curve.  Clearly $Q_k(n)$ is swept out by $A_{2,k}(n)$.

\item If $D$ is a divisor class, by a \emph{dual curve}  to $D$ we mean an effective curve class $C$ with $C\cdot D = 0$.  By Lemma \ref{chamberCurveLemma}, if $C$ is a dual curve to $H-\alpha B$ with $\alpha>0$, then the locus swept out by irreducible curves of class $C$ lies in the stable base locus of $H - \beta B$ for $\beta > \alpha$.
\end{itemize}

In the stable base locus tables that follow, the stable base locus of the chamber spanned by two adjacent divisors is listed in the row between them.  We note that the effective cone is spanned by the first and last listed divisor classes.  In every case, this statement is justified by Theorem \ref{effective}.  We do not list dual curves to the final edge of the cone when these curves are complicated as they are irrelevant to the discussion of the stable base locus; the dual curve can be found in the proof of Theorem \ref{effective}.  When also give geometric descriptions of effective divisors spanning each ray.

\subsection{The walls for $\PP^{2[2]}$} In this example, we work out the stable base locus decomposition of $\PP^{2[2]}$ and the corresponding Bridgeland walls.

The stable base locus decomposition of $\PP^{2[2]}$ is as follows.

\begin{center}
\begin{longtable}{cccc}
\toprule 
Divisor class & Divisor description & Dual curves & Stable base locus\\\midrule
\endfirsthead
\multicolumn{4}{l}{{\small \it continued from previous page}}\\
\toprule
Divisor class & Divisor description & Dual curves & Stable base locus\\\midrule \endhead
\bottomrule \multicolumn{4}{r}{{\small \it continued on next page}} \\ \endfoot
\bottomrule
\endlastfoot
$B$ & $B$ & $C_1(2)$ &\\ &&&$B$\\
$H$ & $H$ & $C(2)$ &\\ &&&$\emptyset$\\
$H-\frac{1}{2}B$ & $D_1(2)$ &  $C_2(2)$
\end{longtable}
\end{center}

\begin{proof}The nef and effective cones of $\PP^{2[2]}$ follow from Proposition \ref{nef} and Theorem \ref{effective}.  The stable base locus in the cone spanned by $H$ and $B$ is described in Proposition \ref{BH}.  For larger $n$, the stable base locus decomposition will have chambers analogous to the chambers here; we will not mention them when justifying the decomposition.\end{proof}

The Bridgeland walls of $\PP^{2[2]}$ are described as follows.
\begin{itemize}
\item There is a unique semi-circular Bridgeland wall with center $x= -\frac{5}{2}$ and radius $\frac{3}{2}$ in the $(s,t)$-plane with $s<0$ and $t >0$ corresponding to the destabilizing object $\OO_{\PP^2}(-1)$.
\end{itemize} 

\begin{proof}
Let $\cI_Z$ be ideal sheaf corresponding to $Z \in \PP^{2[2]}$. Since $Z$  is contained on a line, $\cI_Z$ admits a non-zero map $\OO_{\PP^2}(-1) \rightarrow \cI_Z$. Hence, already every ideal sheaf $\cI_Z$ is destabilized at the wall with $x = - \frac{5}{2}$ arising from $\OO_{\PP^2}(-1)$. By Observation \ref{rank-1}, this is the only rank 1 wall.

Suppose there were any walls given by a rank 2 destabilizing sheaf $\cF$. Then, by Inequality (\ref{main}), the center $x$ of the corresponding wall satisfies $x^2 <  \frac{9}{2}$. Since $\frac{9}{2} < \frac{25}{4}$, by Observation \ref{numerics}, we conclude that any potential higher rank destabilizing wall has to be contained inside the wall given by $\OO_{\PP^2}(-1)$, which already destabilizes all the ideal sheaves. We conclude that when $n=2$, there is a unique destabilizing wall with center $x=-\frac{5}{2}$.
\end{proof}

$\PP^{2[2]}$ admits three birational models corresponding to the stable base locus decomposition. 
\begin{enumerate}
\item The divisor class $H$ induces the Hilbert-Chow morphism $h:\PP^{2[2]} \rightarrow \PP^{2(2)}$. 
\item A divisor $D$ contained in the open cone bounded by $H$ and $H-\frac{B}{2}$ is ample. A sufficiently high multiple of $D$ gives an embedding of $\PP^{2[2]}$.
\item The divisor class $H-\frac{B}{2}$ induces a Mori fibration $\phi: \PP^{2[2]} \rightarrow (\PP^2)^*$. The support of a  scheme with Hilbert polynomial $2$ is contained in a unique line. The morphism $\phi$ maps a scheme $Z \in \PP^{2[2]}$ to the unique line containing $Z$. The fibers of the morphism $\phi$ are equal to $\PP^2 \cong (\PP^1)^{(2)} \cong \PP^{1[2]}$. 
\end{enumerate}

\subsection{The walls for $\PP^{2[3]}$}
 The stable base locus decomposition of  $\PP^{2[3]}$ is as follows.

\begin{center}
\begin{longtable}{cccc}
\toprule 
Divisor class & Divisor description & Dual curves & Stable base locus\\\midrule
\endfirsthead
\multicolumn{4}{l}{{\small \it continued from previous page}}\\
\toprule
Divisor class & Divisor description & Dual curves & Stable base locus\\\midrule \endhead
\bottomrule \multicolumn{4}{r}{{\small \it continued on next page}} \\ \endfoot
\bottomrule
\endlastfoot
$B$ & $B$ & $C_1(3)$ &\\ &&&$B$\\
$H$ & $H$ & $C(3)$ &\\ &&&$\emptyset$\\
$H-\frac{1}{4}B$ & $D_2(3)$ &  $C_3(3)$\\&&&$L_3(3)=E_1(3)$\\
$H-\frac{1}{2}B$ & $E_1(3)$ &   $C_2(3)$
\end{longtable}
\end{center}
\begin{proof}Except for the final chamber spanned by $H-\frac{1}{4}B$ and $H-\frac{1}{2}B$, everything here is analogous to the $n=2$ case.  Since $C_3(3)$ is dual to $H-\frac{1}{4}B$, the stable base locus in this final chamber contains $L_3(3)$.  But since $L_3(3)=E_1(3)$, which is the divisor spanning $H-\frac{1}{2}B$, we conclude the stable base locus is actually $L_3(3)$.\end{proof}

The Bridgeland walls in the $(s,t)$-plane with $s<0$ and $t>0$ are the following two semi-circles $W_x$ with center $(x,0)$ and radius $\sqrt{x^2 -6}$.

\begin{itemize}
\item The rank one wall $W_{-\frac{7}{2}}$ corresponding to the destabilizing object $\OO_{\PP^2}(-1)$.
\item The two overlapping rank one walls $W_{-\frac{5}{2}}$ corresponding to the destabilizing object  $\OO_{\PP^2}(-2)$ or $\cI_{p}(-1)$. This wall also arises as rank two and rank three walls corresponding to destabilizing objects $\OO_{\PP^2}(-2)^{\oplus 2}$ and $\OO_{\PP^2}(-2)^{\oplus 3}$.
\end{itemize}
\begin{proof}
Let $\cI_Z$ be the ideal sheaf of a zero-dimensional scheme of length three. If the scheme is collinear, then there exists a non-zero map $\OO_{\PP^2}(-1) \rightarrow \cI_Z$. These sheaves are destabilized at the wall with center $x= - \frac{7}{2}$. The next rank 1 wall with $k=1$, occurs when $l(W)=1$ and has center $x=-\frac{5}{2}$. A zero-dimensional scheme of length three is always contained on a conic, hence there exists a non-zero map $\OO_{\PP^2}(-2) \rightarrow \cI_Z$. Hence, all $\cI_Z$ are destabilized at the wall with center $x=-\frac{5}{2}$. By Observation \ref{rank-1}, there are no other rank one walls with center  $x<-\frac{5}{2}$.

By Inequality (\ref{main}), the centers of the rank 2 and 3  walls have to satisfy  $$x^2 \leq \frac{27}{4}, x^2 \leq \frac{25}{4}.$$ By Observation \ref{numerics}, we conclude that any potential destabilizing wall arising from a sheaf of rank $r \geq 4$ would be strictly contained inside the wall $W_{-\frac{5}{2}}$. Since at $W_{-\frac{5}{2}}$ all the ideal sheaves are destabilized, the only other potential walls arise from rank $2$ or $3$ sheaves. Notice that in the rank 3 case, any potential wall that is not contained in the interior of the semicircle defined by $W_{-\frac{5}{2}}$ coincides with $W_{-\frac{5}{2}}$.

We now discuss the rank 2 walls. If there is a rank 2 sheaf $\FF$ giving rise to a wall, we obtain the inequality $$2(x + \sqrt{x^2-6}) \leq d(\FF) \leq x- \sqrt{x^2 -6}.$$ We may assume that $\frac{25}{4} < x^2 \leq \frac{27}{4}$. We conclude that $-4 < d(\FF) < -3$. Since $d(\FF)$ is an integer, we conclude that for any potential wall produced by a rank 2 sheaf either coincides with $W_{-\frac{5}{2}}$ or is strictly contained in the semicircle defined by it. 
Since the generic ideal sheaf $\cI_Z$ is generated by three quadratic equations, there exists a non-zero morphism $\OO_{\PP^2}^{\oplus 3}(-2) \rightarrow \cI_Z$.  Hence the wall $W_{-\frac{5}{2}}$ does also arise as a rank two and rank three wall.
\end{proof}

The Hilbert scheme $\PP^{2[3]}$ admits the following three birational models:

\begin{enumerate}
\item The divisor class $H$ induces the Hilbert-Chow morphism $h: \PP^{2[3]} \rightarrow \PP^{2(3)}$.
\item A divisor in the interior of the cone spanned by $H$ and $H-\frac{B}{4}$ is ample and a sufficiently high multiple of $D$ gives an embedding of $\PP^{2[3]}$.
\item The divisor class $H-\frac{B}{4}$ induces a divisorial contraction $\phi: \PP^{2[3]} \rightarrow G(3,6)$ onto its image. The morphism $\phi$ restricted to the divisor $E_1(3)$ maps $E_1(3)$ to $(\PP^2)^*$ with fibers isomorphic to $\PP^3 = \PP^{1(3)} = \PP^{1[3]}$. The resulting model may be interpreted as the moduli space of Bridgeland stable objects $\cM_{s,t}(1,0,-3)$, for a point $(s,t)$ between the walls $W_{-\frac{7}{2}}$ and  $W_{-\frac{5}{2}}$ 
\end{enumerate}

\subsection{The walls for $\PP^{2[4]}$}
The stable base locus decomposition of  $\PP^{2[4]}$ is as follows.

\begin{center}
\begin{longtable}{cccc}
\toprule 
Divisor class & Divisor description & Dual curves & Stable base locus\\\midrule
\endfirsthead
\multicolumn{4}{l}{{\small \it continued from previous page}}\\
\toprule
Divisor class & Divisor description & Dual curves & Stable base locus\\\midrule \endhead
\bottomrule \multicolumn{4}{r}{{\small \it continued on next page}}\\ \endfoot
\bottomrule
\endlastfoot
$B$ & $B$ & $C_1(4)$ &\\ &&&$B$\\
$H$ & $H$ & $C(4)$ &\\ &&&$\emptyset$\\
$H-\frac{1}{6}B$ & $D_3(4)$ &  $C_4(4)$\\&&&$L_4(4)$\\
$H-\frac{1}{4}B$ & $D_2(4)$ &   $C_3(4)$\\&&&$L_3(4)=E_1(4)$\\
$H-\frac{1}{3}B$ & $E_{1}(4)$ & $A_{2,4}(4)$
\end{longtable}
\end{center}

\begin{proof}
The stable base locus for the chamber spanned by $H-\frac{1}{6} B$ and $H-\frac{1}{4}B$ follows from Proposition \ref{DkBase} (a); recall that we showed $C_4(4)$ was dual to $H-\frac{1}{6}B$ and that $L_4(4)$ lied in the base locus of $D_2(4)$.  In the future many chambers will follow from this proposition, and we will not comment about them.  On the other hand the stable base locus for the final chamber follows as in the final chamber for $n=3$.   
\end{proof}

The Bridgeland walls in the $(s,t)$-plane with $s<0$ and $t>0$ are the following three semi-circles $W_x$ with center $(x,0)$ and radius $\sqrt{x^2-8}$:

\begin{itemize}
\item The rank one wall $W_{-\frac{9}{2}}$ corresponding to the destabilizing object $\OO_{\PP^2}(-1)$.
\item The rank one wall $W_{-\frac{7}{2}}$ corresponding to the destabilizing object $\cI_p(-1)$.
\item The rank one wall $W_{-3}$ corresponding to the destabilizing object $\OO_{\PP^2}(-2)$. This wall also arises as a rank two wall corresponding to the destabilizing object $\OO_{\PP^2}(-2)^{\oplus 2}$. 
\end{itemize}

\begin{proof}
Let $\cI_Z$ be the ideal sheaf of a zero-dimensional scheme of length four. The rank one walls with $k=1$ have centers $x = -\frac{9}{2}, -\frac{7}{2}$, corresponding to sheaves of length four that are collinear and sheaves of length four that have a collinear length three subscheme. Since every scheme of length four is contained in a conic, there is a non-zero map $\OO_{\PP^2}(-2) \rightarrow \cI_Z$. Hence, all the ideal sheaves are destabilized at the wall with center $x = -3$.  By Observation \ref{rank-1}, there are no other rank one walls with center $x < -3$.

By Inequality (\ref{main}), the center of a potential rank 2 wall has to satisfy the inequality $x^2 \leq 9$. By Observation \ref{numerics}, there cannot be any potential walls of rank $r \geq 3$ that are not contained in the semicircle defined by $W_{-3}$. Furthermore, any potential rank 2 wall, either has $x=-3$ and coincides with $W_{-3}$ or is strictly contained in the semicircle defined by $W_{-3}$. Since the resolution of a general zero-dimensional scheme of length four has the form $$0 \rightarrow \OO_{\PP^2}(-4) \rightarrow \OO_{\PP^2}(-2) \oplus \OO_{\PP^2}(-2) \rightarrow \cI_Z \rightarrow 0,$$ $W_{-3}$ also occurs as a rank 2 wall corresponding to  $\OO_{\PP^2}(-2) \oplus \OO_{\PP^2}(-2)$. 
\end{proof}

The Hilbert scheme $\PP^{2[4]}$ admits five birational models. 

\begin{enumerate}
\item The divisor class $H$ gives rise to the Hilbert-Chow morphism $h: \PP^{2[4]} \rightarrow \PP^{2(4)}$.
\item A divisor contained in the open cone bounded by $H$ and $H- \frac{B}{6}$ is ample and a sufficiently high multiple gives an embedding of $\PP^{2[4]}$.
\item The divisor $H-\frac{B}{6}$ induces a small contraction $\phi: X_4 \rightarrow G(6,10)$ onto its image. The morphism $\phi$ is an isomorphism outside the locus of schemes supported on a line. The morphism $\phi$ maps the locus of schemes in $\PP^{2[4]}$ supported on a line to $(\PP^2)^*$ with fibers isomorphic to $\PP^4 = \PP^{1(4)} = \Hilb_4(\PP^1)$. 
\item $X_4$ admits a flip $X_4^{fl}$ and the model for a divisor contained in the open cone bounded by $H-\frac{B}{6}$ and $H-\frac{B}{4}$ is $X_4^{fl}$. This flip $X_4^{fl}$ can be interpreted as the coarse moduli scheme of Bridgeland stable objects $\cM_{s,t}(1,0,-4)$ for $(s,t)$ between the walls $W_{-\frac{9}{2}}$ and $W_{-\frac{7}{2}}$. 
\item The divisor $H-\frac{B}{4}$ induces a divisorial contraction $\phi: X_4^{fl} \rightarrow G(2,6)$. The morphism $\phi$ is an isomorphism in the complement of $E_1(4)$. 
\end{enumerate}

\subsection{The walls for $\PP^{2[5]}$} The stable base locus decomposition of  $\PP^{2[5]}$ is as follows.

\begin{center}
\begin{longtable}{cccc}
\toprule 
Divisor class & Divisor description & Dual curves & Stable base locus\\\midrule
\endfirsthead
\multicolumn{4}{l}{{\small \it continued from previous page}}\\
\toprule
Divisor class & Divisor description & Dual curves & Stable base locus\\\midrule \endhead
\bottomrule \multicolumn{4}{r}{{\small \it continued on next page}} \\ \endfoot
\bottomrule
\endlastfoot
$B$ & $B$ & $C_1(5)$ &\\ &&&$B$\\
$H$ & $H$ & $C(5)$ &\\ &&&$\emptyset$\\
$H-\frac{1}{8}B$ & $D_4(5)$ &  $C_5(5)$\\&&&$L_5(5)$\\
$H-\frac{1}{6}B$ & $D_3(5)$ &   $C_4(5)$\\&&&$L_4(5)$\\
$H-\frac{1}{4}B$ & $D_2(5)$ &  $C_3(5)$
\end{longtable}
\end{center}

\begin{proof}
The last two chambers here both follow from Proposition \ref{DkBase} (a).
\end{proof}

The Bridgeland walls in the $(s,t)$-plane are the following three semi-circles $W_x$ with center at $(x,0)$ and radius $\sqrt{x^2-10}$. 

\begin{itemize}
\item The rank one wall $W_{-\frac{11}{2}}$ corresponding to the destabilizing object $\OO_{\PP^2}(-1)$.
\item The rank one wall $W_{-\frac{9}{2}}$ corresponding to the destabilizing object $\cI_p(-1)$.
\item The two coinciding rank one walls $W_{-\frac{7}{2}}$ corresponding to destabilizing objects $\OO_{\PP^2}(-2)$ and $\cI_{Z'}(-1)$, where $Z'$ is a zero-dimensional scheme of length two.
\end{itemize}

\begin{proof}
Let $\cI_Z$ be the ideal sheaf of a zero-dimensional  scheme of length five. The rank one walls with $k=1$ have centers $x= -\frac{11}{2}, -\frac{9}{2}, -\frac{7}{2}$ corresponding to collinear schemes of length five and schemes of length five that have collinear subschemes of lengths four and three, respectively.  Since every scheme of length five is contained in a conic, there exists a non-zero map $\OO_{\PP^2}(-2) \rightarrow \cI_Z$. Hence, every ideal sheaf $\cI_Z$ is destabilized at the rank one wall with $k=2$ and center $x= -\frac{7}{2}$. By Observation \ref{rank-1}, there are no other potential rank one walls with $x< -\frac{7}{2}$.

By Inequality (\ref{main}), the center of a potential rank 2 wall has to satisfy the inequality $x^2 \leq \frac{45}{4}$. Since $\frac{45}{4} < \frac{49}{4}$, by Observation \ref{numerics}, we conclude that any potential wall of rank bigger than one has to be contained in the semicircle defined by $W_{-\frac{7}{2}}$.
\end{proof}

\subsection{The walls for $\PP^{2[6]}$}
The stable base locus decomposition of  $\PP^{2[6]}$ is as follows.
\begin{center}
\begin{longtable}{cccc}
\\ \toprule 
Divisor class & Divisor description & Dual curves & Stable base locus\\\midrule
\endfirsthead
\multicolumn{4}{l}{{\small \it continued from previous page}}\\
\toprule
Divisor class & Divisor description & Dual curves & Stable base locus\\\midrule \endhead
\bottomrule \multicolumn{4}{r}{{\small \it continued on next page}} \\ \endfoot
\bottomrule
\endlastfoot
$B$ & $B$ & $C_1(6)$ &\\ &&&$B$\\
$H$ & $H$ & $C(6)$ &\\ &&&$\emptyset$\\
$H-\frac{1}{10}B$ & $D_5(6)$ &  $C_6(6)$\\&&&$L_6(6)$\\
$H-\frac{1}{8}B$ & $D_4(6)$ &   $C_5(6)$\\&&&$L_5(6)$\\
$H-\frac{1}{6}B$ & $D_3(6)$ &  $C_4(6)$\\&&&  $L_4(6)$\\
$H-\frac{1}{5}B$ & $D_{T_{\PP^2}(1)}(6)$&   $A_{2,6}(6)$\\&&&$Q_6(6) = E_2(6)$\\
$H-\frac{1}{4}B$ & $E_2(6)$&  $C_3(6)$
\end{longtable}
\end{center}
\begin{proof}
Since $A_{2,6}(6)$ is dual to $H-\frac{1}{5}B$, the locus $Q_6(6)= E_2(6)$ lies in the stable base locus of the chamber $(H-\frac{1}{5}B,H-\frac{1}{4}B]$.  Since $H-\frac{1}{4}B$ is spanned by $E_2(6)$, we conclude that this divisor actually is the stable base locus in this chamber.

Observe that $L_4(6)$ is contained in the stable base locus of the chamber spanned by $H-\frac{1}{6}B$ and $H-\frac{1}{5}B$ since $C_4(6)$ is dual to $H-\frac{1}{6}B$.  By Lemma \ref{stableContainLemma}, the stable base locus of this chamber is contained in $Q_6(6)$, the stable base locus of the final chamber.  Note that $D_{T_{\PP^2}(1)}(6)$ spans the edge $H-\frac{1}{5}B$, and that by Proposition \ref{interpClass} its stable base locus lies in the locus of schemes of length $6$ which fail to impose independent conditions on sections of $T_{\PP^2}(1)$.  It therefore suffices to show that any $Z\in Q_6(6)\setminus L_4(6)$ imposes independent conditions on sections of $T_{\PP^2}(1)$---this will also prove that $T_{\PP^2}(1)$ actually satisfies interpolation for $6$ points.  This follows immediately from Lemma \ref{tangent}, to follow.
\end{proof}

\begin{lemma}\label{tangent}
Let $C$ be a conic curve, possibly singular or nonreduced.  If $Z\subset C$ is a zero-dimensional subscheme of length $2k$ which contains no linear subscheme of length $k+1$, then $Z$ imposes independent conditions on sections of $T_{\PP^2}(k-2)$; furthermore, every section of $T_{\PP^2}(k-2) \otimes \cI _Z$ vanishes along $C$.
\end{lemma}
\begin{proof}
Since $h^0(\OO_C(k)) = 2k+1$ we have $h^0(\OO_C(k)\otimes \cI_Z)\geq 1$, so there is a curve of degree $k$ which contains $Z$ but not $C$.  Even when $C$ is singular or nonreduced, it is possible to choose this curve to have zero-dimensional intersection with $C$  since $Z$ does not have a linear subscheme of length $k+1$.  In the singular case, $Z$ does not meet the singularity, so the degree $k$ curve can be taken to be a union of $k$ lines meeting $C$ properly.  On the other hand, if $C=2L$ is nonreduced, then the scheme residual to $Z\cap L$ in $Z$ is a subscheme of $L$ of length $k$.  Every curve of degree $k-1$ which contains such a scheme contains $L$, so any curve of degree $k$ which contains $Z$ but not $2L$ must not contain $L$.

We conclude that $Z$ is the complete intersection of $C$ and a curve of degree $k$, and thus that $\cI_Z|_C \cong \OO_C(-k)$.  Now consider the exact sequences $$0\to T_{\PP^2}(k-4)\to T_{\PP^2}(k-2)\otimes \cI_Z \to (T_{\PP^2}(k-2)\otimes \cI_Z)|_C = T_{\PP^2}(-2)|_C \to 0$$ $$0\to T_{\PP^2}(-4)\to T_{\PP^2}(-2)\to T_{\PP^2}(-2)|_C\to 0.$$  Since $h^0(T_{\PP^2}(-2)) = h^1(T_{\PP^2}(-4)) = 0$ we have $h^0(T_{\PP^2}(-2)|_C) = 0$, so the map $$H^0(T_{\PP^2}(k-4))\to H^0(T_{\PP^2}(k-2)\otimes \cI_Z)$$ is an isomorphism.  But $$h^0(T_{\PP^2}(k-2)) -4k = h^0(T_{\PP^2}(k-4)),$$ so $Z$ imposes the required number $4k$ of conditions on sections of $T_{\PP^2}(k-2)$.
\end{proof}

The Bridgeland walls in the $(s,t)$-plane are the following five semi-circles $W_x$ with center at $(x,0)$ and radius $\sqrt{x^2-12}$. 

\begin{itemize}
\item The rank one wall $W_{-\frac{13}{2}}$ corresponding to the destabilizing object $\OO_{\PP^2}(-1)$.
\item The rank one wall $W_{-\frac{11}{2}}$ corresponding to the destabilizing object $\cI_p(-1)$.
\item The rank one wall $W_{-\frac{9}{2}}$ corresponding to the destabilizing object $\cI_{Z'}(-1)$, where $Z'$ is a zero-dimensional scheme of length two.
\item The rank one wall $W_{-4}$ corresponding to the destabilizing object $\OO_{\PP^2}(-2)$.
\item The rank coinciding rank one walls $W_{-\frac{7}{2}}$ corresponding to the destabilizing objects $\OO_{\PP^2}(-3)$, $\cI_p(-2)$,  and $\cI_{Z'}(-1)$, where $Z'$ is a zero-dimensional scheme of length three.
\end{itemize}

\begin{proof}
Let $\cI_Z$ be the ideal sheaf of a scheme of length six. The rank one walls with $k=1$ have centers $x= -\frac{13}{2}, -\frac{11}{2}, -\frac{9}{2}, -\frac{7}{2}$ corresponding to schemes of length six that have a subscheme of length six, five, four or three supported along a line, respectively. The rank one walls with $k=2$ have centers $x=-4, -\frac{7}{2}$ corresponding to schemes of length six that have a subscheme of length six or five supported along a conic, respectively. Finally, the rank one wall with $k=3$ has center $x=-\frac{7}{2}$. Since every scheme of length six is contained in a cubic curve, there exists a non-zero map $\OO_{\PP^2}(-3) \rightarrow \cI_Z$. Hence, every ideal sheaf is destabilized at the wall $W_{-7/2}$. By Observation \ref{rank-1}, there are no other rank one walls with $x< - \frac{7}{2}$. 

By inequality (\ref{main}), the centers of a potential rank 2, 3 or 4 walls have to satisfy the inequality $x^2 \leq \frac{54}{4}, x^2 \leq \frac{50}{4}, x^2 \leq \frac{49}{4}$. By Observation \ref{numerics}, we conclude that any potential wall defined by a destabilizing sheaf  of rank $r \geq 5$ has to be contained in the semicircle defined by $W_{-\frac{7}{2}}$. Similarly, any potential wall defined by a destabilizing sheaf of rank 4 either coincides with $W_{-\frac{7}{2}}$ or is strictly contained in the semicircle defined by $W_{-\frac{7}{2}}$.

We now have to analyze the rank 2 and 3 walls. For rank 2 walls, by inequality (\ref{main}), we have that 
$2(x + \sqrt{x^2 -12}) \leq d(\FF) \leq x - \sqrt{x^2 -12}$. We may assume that $$\frac{49}{4} < x^2 \leq \frac{54}{4}.$$ Hence, $-6 < d(\FF) < -4$. We conclude that any destabilizing rank 2 sheaf has degree $d(\FF) = -5$. Now we can bound the second Chern character of $\FF$. We have that the center of the semicircle is given by $$- \frac{7}{2} > x= \frac{ch_2(\FF) +12}{-5} \geq - \frac{3 \sqrt{6}}{2}.$$ Since $\ch_2(\FF)$ is a half integer, we conclude that $ch_2(\FF) = 6$. In other words, $\FF$ is a coherent sheaf of rank $2$ with $c_1(\FF) = -5$ and $\ch_2(\FF) = 6$.  This is not actually possible, since then $c_2(\FF)$ would not be an integer.  Thus there are no rank 2 walls.

Next, we work out the rank three walls. By inequality (\ref{main}), we have that $3 (x + \sqrt{x^2 -12}) \leq d(\FF) \leq 2(x - \sqrt{x^2 -12})$. We may assume that $$\frac{49}{4} < x^2 \leq \frac{50}{4}.$$ Hence, we conclude that $-9 < d(\FF) < -8$. Since the degree is an integer, we conclude that the only potential rank three walls either coincide with $W_{-\frac{7}{2}}$ or are strictly contained in the semicircle defined by $W_{-\frac{7}{2}}$.
\end{proof}

\subsection{The walls for $\PP^{2[7]}$}
The stable base locus decomposition of  $\PP^{2[7]}$ is as follows.

\begin{center}
\begin{longtable}{cccc}
\toprule 
Divisor class & Divisor description & Dual curves & Stable base locus\\\midrule
\endfirsthead
\multicolumn{4}{l}{{\small \it continued from previous page}}\\
\toprule
Divisor class & Divisor description & Dual curves & Stable base locus\\\midrule \endhead
\bottomrule \multicolumn{4}{r}{{\small \it continued on next page}} \\ \endfoot
\bottomrule
\endlastfoot
$B$ & $B$ & $C_1(7)$ &\\ &&&$B$\\
$H$ & $H$ & $C(7)$ &\\ &&&$\emptyset$\\
$H-\frac{1}{12}B$ & $D_6(7)$ &  $C_7(7)$\\&&&$L_7(7)$\\
$H-\frac{1}{10}B$ & $D_5(7)$ &   $C_6(7)$\\&&&$L_6(7)$\\
$H-\frac{1}{8}B$ & $D_4(7)$ &  $C_5(7)$\\&&& $L_5(7)$\\
$H-\frac{1}{6}B$ & $D_3(7)$ &   $C_4(7),A_{2,7}(7)$\\&&& $L_4(7) \cup Q_{7}(7)$\\
$H-\frac{1}{5}B$ & $D_{T_{\PP^2}(1)}(7)$  &  $A_{2,6}(7)$\\&&& $Q_6(7)=E_2(7)$\\
$H-\frac{5}{24}B$ & $E_2(7)$ 
\end{longtable}
\end{center}
\begin{proof}
All chambers besides the chamber spanned by $H-\frac{1}{6}B$ and $H-\frac{1}{5}B$ are completely analogous to previous cases.  In this chamber, we need only show that if $Z\in Q_6(7)\setminus (L_4(7)\cup Q_7(7))$ then $Z$ imposes independent conditions on sections of $T_{\PP^2}(1)$.  This follows easily from Lemma \ref{tangent}.
\end{proof}

The case $n=7$ is interesting because we see the first example of a collapsing rank two wall. 
The Bridgeland walls in the $(s,t)$-plane are the following six semi-circles $W_x$ with center at $(x,0)$ and radius $\sqrt{x^2-14}$. 

\begin{itemize}
\item The rank one wall $W_{- \frac{15}{2}}$ corresponding to the destabilizing object $\OO_{\PP^2}(-1)$.
\item The rank one wall $W_{-\frac{13}{2}}$ corresponding to the destabilizing object $\cI_{p}(-1)$.
\item  The rank one wall $W_{-\frac{11}{2}}$ corresponding to the destabilizing object $\cI_{Z'}(-1),$ where $Z'$ is a scheme of length two.
\item The two coinciding rank one walls $W_{- \frac{9}{2}}$ corresponding to the destabilizing objects $\OO_{\PP^2}(-2)$ and $\cI_{Z'}(-1),$ where $Z'$ is a scheme of length 3.
\item The rank one wall $W_{-4}$ corresponding to the destabilizing object $\cI_{p}(-2)$.
\item The rank two wall $W_{- \frac{39}{10}}$ corresponding to the destabilizing object $T_{\PP^2}(-4)$.
\end{itemize}

\begin{proof}
By Gaeta's Theorem \cite{eisenbud}, the syzygies for a general zero dimensional scheme of length seven is given by 
$$\OO_{\PP^2}(-5) \oplus \OO_{\PP^2}(-4) \rightarrow \OO_{\PP^2}(-3)^{\oplus 3} \rightarrow \cI_Z \rightarrow 0.$$ By the twisting the Euler sequence by $\OO_{\PP^2}(-4)$, we obtain the short exact sequence of sheaves $$0 \rightarrow \OO_{\PP^2}(-4) \rightarrow \OO_{\PP^2}(-3)^{\oplus 3} \rightarrow T_{\PP^2}(-4).$$ Hence, $\cI_Z$ fits in a short exact sequence $$0 \rightarrow T_{\PP^2}(-4) \rightarrow \cI_Z \rightarrow \OO_{\PP^2}(-5)[1] \rightarrow 0$$ in the category.
Since $c_1(T_{\PP^2}(-4)) = - 5 L $ and $\ch_2(T_{\PP^2}(-4)) = \frac{11}{2}$, the sheaf $T_{\PP^2}(-4)$ gives rise to the wall $W_{- \frac{39}{10}}$. We conclude that the wall $W_{- \frac{39}{10}}$ is a collapsing wall in the sense that all the ideal sheaves $\cI_Z$ are destabilized by the time we reach $W_{- \frac{39}{10}}$.

Let $\cI_Z$ be the ideal sheaf of a scheme of length seven. The rank one walls with $k=1$ have centers $x= -\frac{15}{2}, -\frac{13}{2}, -\frac{11}{2}, -\frac{9}{2}$ corresponding to subschemes of length seven that have a collinear subscheme of length seven, six, five and four, respectively. The rank one walls with $k=2$ have centers $x= -\frac{9}{2}, -4$ corresponding to schemes of length seven that have a subscheme of length seven and six, respectively, supported along a conic. Finally, a potential rank one wall with $k=3$ has center $x= -\frac{23}{6}$. However, $- \frac{39}{10} < - \frac{23}{6}$. Hence, the potential wall $W_{- \frac{9}{2}}$ is strictly contained in $W_{- \frac{39}{10}}$ and all the ideal sheaves are destabilized before reaching the wall $W_{- \frac{9}{2}}$. By Observation \ref{rank-1}, these are the only rank one walls with $x< -\frac{39}{10}$.

By Inequality (\ref{main}), the centers of the potential rank 2 and rank 3 walls have to satisfy $$x^2 \leq \frac{63}{4} \ \ \mbox{and} \ \ x^2 \leq \frac{175}{12},$$ respectively. Since $\frac{175}{12} < (\frac{39}{10})^2$, by observation \ref{numerics}, we conclude that any potential wall of rank three or higher is contained in the semicircle defined by $W_{-\frac{39}{10}}$. 

We now have to analyze the rank 2 walls. We have the inequality $$2(x + \sqrt{x^2 -14}) \leq d(\FF) \leq x - \sqrt{x^2 -14}.$$ Since $(\frac{39}{10})^2 < x^2 \leq \frac{63}{4}$, we conclude that $-28/5< d(\FF) < -5$.  Since $d(\FF)$ is an integer, no such walls are possible.
\end{proof}

\subsection{The walls for $\PP^{2[8]}$}
The stable base locus decomposition of  $\PP^{2[8]}$.

\begin{center}
\begin{longtable}{cccc}
\toprule 
Divisor class & Divisor description & Dual curves & Stable base locus\\\midrule
\endfirsthead
\multicolumn{4}{l}{{\small \it continued from previous page}}\\
\toprule
Divisor class & Divisor description & Dual curves & Stable base locus\\\midrule \endhead
\bottomrule \multicolumn{4}{r}{{\small \it continued on next page}} \\ \endfoot
\bottomrule
\endlastfoot
$B$ & $B$ & $C_1(8)$ &\\ &&&$B$\\
$H$ & $H$ & $C(8)$ &\\ &&&$\emptyset$\\
$H-\frac{1}{14}B$ & $D_7(8)$ &  $C_8(8)$\\&&&$L_8(8)$\\
$H-\frac{1}{12}B$ & $D_6(8)$ &   $C_7(8)$\\&&&$L_7(8)$\\
$H-\frac{1}{10}B$ & $D_5(8)$ &  $C_6(8)$\\&&& $L_6(8)$\\
$H-\frac{1}{8}B$ & $D_4(8)$ &   $C_5(8)$\\&&& $L_5(8)$\\
$H-\frac{1}{7}B$ & $D_{T_{\PP^2}(2)}(8)$ &  $A_{2,8}(8)$\\&&& $L_5(8)\cup Q_8(8)$\\
$H-\frac{1}{6}B$ & $D_3(8)$ & $C_4(8),A_{2,7}(8)$\\ &&& $L_4(8)\cup Q_7(8) $\\
$H-\frac{3}{16}B$ & $D_{\mathrm{coker}(\OO(1)^2\to \OO(2)^5)}(8)$ & 
\end{longtable}
\end{center}
\begin{proof}
The last three chambers are interesting here.  To get the easiest one out of the way, consider the chamber spanned by $H-\frac{1}{7}B$ and $H-\frac{1}{6}B$.  Since $C_5(8)$ is dual to $H-\frac{1}{8}B$, every divisor $H-\alpha B$ with $\alpha > \frac{1}{8}$ contains $L_5(8)$ in its stable base locus.  Since $A_{2,8}(8)$ is dual to $H-\frac{1}{7}B$, the stable base locus in the chamber spanned by $H-\frac{1}{7}$ and $H-\frac{1}{6}B$ contains $L_5(8)\cup Q_8(8)$.  But Proposition \ref{DkBase} (b) shows that the base locus of $D_3(8)$ is actually contained in $L_5(8)\cup Q_8(8)$, so we conclude $L_5(8)\cup Q_8(8)$ is the stable base locus in this chamber.

For the final chamber in the cone, spanned by $H-\frac{1}{6}B$ and $H-\frac{3}{16}B$, first observe that since the curves $C_4(8)$ and $A_{2,7}(8)$ are dual to $H-\frac{1}{6}B$ the stable base locus is contained in $L_4(8)\cup Q_7(8)$. For the other containment, we make use of another description of a divisor spanning the edge of the effective cone.  A general collection of $8$ points lies on a unique pencil of cubic curves, and so determines a $9$th point given as the final base point of this pencil.  Fixing a line $\ell$ in $\PP^2$, we obtain a divisor in $\PP^{2[8]}$ described as the locus where the $9$th point lies on $\ell$.  An elementary calculation shows this divisor spans the ray $H-\frac{3}{16}B$.  

Since varying the line $\ell$ does not change the class of the divisor, the stable base locus of the class $H-\frac{3}{16}B$ lies in the locus of schemes $Z$ of length $8$ which either fail to impose independent conditions on cubics or such that the base locus of the pencil containing $Z$ is positive dimensional.  In the former case, we saw in Lemma \ref{2dplus1lemma} (b) that $Z$ lies in $L_5(8)\cup Q_8(8)$, which is contained in $L_4(8)\cup Q_7(8)$.  On the other hand, suppose $Z$ imposes independent conditions on cubics but that the base locus of the pencil containing $Z$ is positive dimensional, with a reduced, irreducible curve $C$ in its base locus. Since the restriction map $$H^0(\OO_{\PP^2}(3))\to H^0(\OO_C(3))$$ is surjective, this is only possible if $Z\cap C$ has length at least $h^0(\OO_C(3))$.  As $C$ could be a line or a conic, this means $Z$ lies in $L_4(8)\cup Q_7(8)$.

For the chamber spanned by $H-\frac{1}{8}B$ and $H-\frac{1}{7}B$,  it suffices to show that if $Z\in Q_8(8)\setminus L_5(8)$ then $Z$ imposes independent conditions on sections of $T_{\PP^2}(2)$; this follows from Lemma \ref{tangent}.
\end{proof}

The Bridgeland walls in the $(s,t)$-plane are the following six semi-circles $W_x$ with center at $(x,0)$ and radius $\sqrt{x^2-16}$. 

\begin{itemize}
\item The rank one wall $W_{- \frac{17}{2}}$ corresponding to the destabilizing object $\OO_{\PP^2}(-1)$.
\item The rank one wall $W_{-\frac{15}{2}}$ corresponding to the destabilizing object $\cI_{p}(-1)$.
\item  The rank one wall $W_{-\frac{13}{2}}$ corresponding to the destabilizing object $\cI_{Z'}(-1),$ where $Z'$ is a scheme of length two.
 \item  The rank one wall $W_{-\frac{11}{2}}$ corresponding to the destabilizing object $\cI_{Z'}(-1),$ where $Z'$ is a scheme of length three.
 \item The rank one wall $W_{-5}$ corresponding to the destabilizing object $\OO_{\PP^2}(-2)$
\item The coinciding rank one walls $W_{- \frac{9}{2}}$ corresponding to the destabilizing objects $\cI_p(-2)$ and $\cI_{Z'}(-1),$ where $Z'$ is a scheme of length four.
\item The rank one wall $W_{-\frac{25}{6}}$ corresponding to the destabilizing object $\OO_{\PP^2}(-3)$. This wall coincides with the rank two wall corresponding to the destabilizing object $\OO_{\PP^2}(-3)^{\oplus 2}$.
\end{itemize}

\begin{proof}

Let $\cI_Z$ be the ideal sheaf of a scheme of length eight.  The rank one walls with $k=1$ have centers $x= -\frac{17}{2}, -\frac{15}{2}, -\frac{13}{2}, -\frac{11}{2}, -\frac{9}{2}$ corresponding to subschemes of length eight that have a collinear subscheme of length eight, seven, six, five and four, respectively. The rank one walls with $k=2$ have centers $x= -5, -\frac{9}{2}$ corresponding to schemes of length eight that have a subscheme of length eight and  seven, respectively, supported along a conic. Finally, the rank one wall with $k=3$ has center $x= -\frac{25}{6}$. Since every scheme of length seven is contained in a curve of degree three, we have a non-zero map $\OO_{\PP^2}(-3) \rightarrow \cI_Z$. Hence, all the ideal sheaves $\cI_Z$ are destabilized at the wall $W_{-\frac{25}{6}}$. As in the previous cases, these are the only potential rank one Bridgeland walls  with $x < - \frac{25}{6}$.

By inequality (\ref{main}), the centers of the potential rank 2 and rank 3 walls have to satisfy $$x^2 \leq 18 \ \ \mbox{and} \ \ x^2 \leq \frac{50}{3},$$ respectively. Since $\frac{50}{3} < (\frac{25}{6})^2$, by observation \ref{numerics}, we conclude that any potential walls defined by sheaves of rank three or higher are contained in the semicircle defined by $W_{-\frac{25}{6}}$. We now analyze the potential rank 2  walls more closely.  For rank 2 walls we have the inequality 
$$2(x + \sqrt{x^2 -16}) \leq d(\FF) \leq x - \sqrt{x^2 -16}.$$ Since $\frac{625}{36} < x^2 \leq \frac{49}{3}$, we conclude that $-6 < d(\FF) < - 32/6$. Since $d(\FF)$ is an integer, we conclude that the potential rank 2 walls either are contained in the semicircle defined by $W_{-\frac{25}{6}}$ or coincide with it. In fact, when a scheme of length 8 is contained in the complete intersection of two cubics, we get a non-zero map $\OO_{\PP^2}(-3) \oplus \OO_{\PP^2}(-3) \rightarrow \cI_Z$ destabilizing the ideal sheaf. The corresponding semicircle coincides with $W_{-25/6}$.
\end{proof}

\subsection{The walls for $\PP^{2[9]}$}
 The stable base locus decomposition of  $\PP^{2[9]}$.

\begin{center}
\begin{longtable}{cccc}
\toprule 
Divisor class & Divisor description & Dual curves & Stable base locus\\\midrule
\endfirsthead
\multicolumn{4}{l}{{\small \it continued from previous page}}\\
\toprule
Divisor class & Divisor description & Dual curves & Stable base locus\\\midrule \endhead
\bottomrule \multicolumn{4}{r}{{\small \it continued on next page}} \\ \endfoot
\bottomrule
\endlastfoot
$B$ & $B$ & $C_1(9)$ &\\ &&&$B$\\
$H$ & $H$ & $C(9)$ &\\ &&&$\emptyset$\\
$H-\frac{1}{16}B$ & $D_8(9)$ & $C_9(9)$\\ &&& $L_9(9)$\\
$H-\frac{1}{14}B$ & $D_7(9)$ &  $C_8(9)$\\&&&$L_8(9)$\\
$H-\frac{1}{12}B$ & $D_6(9)$ &   $C_7(9)$\\&&&$L_7(9)$\\
$H-\frac{1}{10}B$ & $D_5(9)$ &  $C_6(9)$\\&&& $L_6(9)$\\
$H-\frac{1}{8}B$ & $D_4(9)$ &   $C_5(9),A_{2,9}(9)$\\&&&$L_5(9)\cup Q_9(9)$\\
$H-\frac{1}{7}B$ & $D_{T_{\PP^2}(2)}(9)$ &  $A_{2,8}(9)$\\&&& $L_5(9) \cup Q_8(9)$\\
$H-\frac{1}{6}B$ & $D_3(9)$ 
\end{longtable}
\end{center}
\begin{proof}
The only part which is not routine at this point is to determine the stable base locus in the final chamber, spanned by $H-\frac{1}{7}B$ and $H-\frac{1}{6}B$.  It clearly contains $L_5(9)\cup Q_8(9)$ and is contained in the locus of schemes of degree $9$ which fail to impose independent conditions on curves of degree $3$.  One checks that if a scheme $Z$ of degree $9$ fails to impose independent conditions on curves of degree $3$ but does not lie in $L_5(9)\cup Q_8(9)$ then it is a complete intersection of two cubics.  We must show that if a scheme $Z$ is a complete intersection of two cubics then it does not lie in the stable base locus of this chamber.

To do this, we need to introduce some more effective divisors on $\PP^{2[9]}$.  Consider a vector bundle $E_k$ given by a general resolution $$0\to \OO_{\PP^2}(1)^{k+1}\to \OO_{\PP^2}(2)^{2k+1}\to E_k\to 0.$$ Then $h^0(E_{k}) = 9k+3$.  We will see soon that $E_k$ satisfies interpolation for $9$ points.  Thus the divisor $D_{E_k}(9)$ will have class spanning the ray $H - \frac{k}{6k+2}B$.  As $k\to \infty$ this ray tends to $H-\frac{1}{6}B$, so by Proposition \ref{interpClass} it suffices to show that if $Z$ is a complete intersection of cubics then it imposes independent conditions on sections of $E_k$.

If $Z$ is a complete intersection of two cubics then its ideal sheaf admits a resolution $$0\to \OO_{\PP^2}(-6)\to \OO_{\PP^2}(-3)^{\oplus 2}\to \cI_Z\to 0,$$ so we have an exact sequence $$0\to E_k(-6)\to E_k(-3)^{\oplus 2}\to E_k\otimes \cI_Z\to 0.$$ It is straightforward to check from the defining sequence for $E_{k}$ that $H^0(E_k(-3)) = H^1(E_k(-3)) = 0$, so $H^0(E_k\otimes \cI_Z)\cong H^1(E_k(-6)).$ We can also check that $h^1(E_{k}(-6))=3$, so $Z$ imposes the full $9k$ independent conditions on sections of $E_k$.
\end{proof}

The Bridgeland walls in the $(s,t)$-plane are the following seven semi-circles $W_x$ with center at $(x,0)$ and radius $\sqrt{x^2-18}$. 

\begin{itemize}
\item The rank one wall $W_{- \frac{19}{2}}$ corresponding to the destabilizing object $\OO_{\PP^2}(-1)$.
\item The rank one wall $W_{-\frac{17}{2}}$ corresponding to the destabilizing object $\cI_{p}(-1)$.
\item  The rank one wall $W_{-\frac{15}{2}}$ corresponding to the destabilizing object $\cI_{Z'}(-1),$ where $Z'$ is a scheme of length two.
 \item  The rank one wall $W_{-\frac{13}{2}}$ corresponding to the destabilizing object $\cI_{Z'}(-1),$ where $Z'$ is a scheme of length three.
\item The coinciding rank one walls $W_{- \frac{11}{2}}$ corresponding to the destabilizing objects $\OO_{\PP^2}(-2)$ and $\cI_{Z'}(-1),$ where $Z'$ is a scheme of length four.
 \item The rank one wall $W_{-5}$ corresponding to the destabilizing object $\cI_p(-2)$

\item The coinciding rank one walls $W_{-\frac{9}{2}}$ corresponding to the destabilizing objects $\OO_{\PP^2}(-3)$, $\cI_{Z'}(-2)$, where $Z'$ has length two, and $\cI_{Z''}(-1)$, where $Z''$ has degree five.
\end{itemize}

\begin{proof}
Let $\cI_Z$ be the ideal sheaf of a scheme of length nine. The rank one walls with $k=1$ have centers $x=-\frac{19}{2}, -\frac{17}{2}, -\frac{15}{2}, -\frac{13}{2}, -\frac{11}{2}, -\frac{9}{2}$ corresponding to schemes of length nine that have collinear subschemes of length nine, eight, seven, six, five and four, respectively. The rank one walls with $k=2$ have centers $x= -\frac{11}{2}, -5, -\frac{9}{2}$ corresponding to schemes of length nine that have a subscheme of length nine, eight and seven, respectively, supported along a conic. Finally, the rank one wall with $k=3$ has center $x= - \frac{9}{2}$. Since every scheme of length nine is supported along a cubic curve, there is a non-zero map $\OO_{\PP^2}(-3) \rightarrow \cI_Z$. Hence, every ideal sheaf $\cI_Z$ is destabilized at the wall $W_{-\frac{9}{2}}$. By Observation \ref{rank-1}, these are all the rank one walls with $x < - \frac{9}{2}$. 

By inequality (\ref{main}), the centers of potential rank 2 walls have to satisfy the inequality $x^2 \leq 81/4$. Since $81/4 = (-9/2)^2$, by Observation \ref{rank-1}, we conclude that any wall defined by a higher rank sheaf either is contained in the semicircle defined by $W_{-\frac{9}{2}}$ or coincides with $W_{-\frac{9}{2}}$. If $Z$ is a complete intersection of two cubics, then $\cI_Z$ admits a map $\OO_{\PP^2}(-3) \oplus \OO_{\PP^2}(-3) \rightarrow \cI_Z$. Hence, the wall $W_{-\frac{9}{2}}$ also occurs as a rank two wall corresponding to $\OO_{\PP^2}(-3)^{\oplus 2}$. 
\end{proof}

\end{document}